\documentclass[10pt, a4paper]{amsart}

\usepackage{array}
\usepackage{bm} 
\usepackage{amstext,amssymb,mathtools,esint,tabularx}
\usepackage{microtype} 
\usepackage{color,graphicx,subfigure,float}
\usepackage{comment}
\usepackage{fixltx2e} 
\usepackage{enumitem}
\usepackage{tikz,pgfplots,pgfplotstable} 
\pgfplotsset{compat=newest}
\usetikzlibrary{calc,3d,arrows,patterns,decorations}
\tikzset{>=stealth'}
\usetikzlibrary{positioning} 
\usepackage[abs]{overpic}
\usepackage[outline]{contour}
\contourlength{0.6pt}
\usetikzlibrary{decorations.text}
\tikzset{text deco/.style={postaction={decorate, decoration={text along path,#1}}}}
\usepackage{yfonts}

\usepackage{ifpdf} \ifpdf  \fi 
\usepackage{cleveref}
\makeatletter 
\pgfdeclarepatternformonly[\LineSpace]{my north east lines}{\pgfqpoint{-1pt}{-1pt}}{\pgfqpoint{\LineSpace}{\LineSpace}}{\pgfqpoint{\LineSpace}{\LineSpace}}
{
    \pgfsetcolor{\tikz@pattern@color} 
    \pgfsetlinewidth{0.4pt}
    \pgfpathmoveto{\pgfqpoint{0pt}{0pt}}
    \pgfpathlineto{\pgfqpoint{\LineSpace + 0.1pt}{\LineSpace + 0.1pt}}
    \pgfusepath{stroke}
}

\pgfdeclarepatternformonly[\LineSpace]{my north west lines}{\pgfqpoint{-1pt}{-1pt}}{\pgfqpoint{\LineSpace}{\LineSpace}}{\pgfqpoint{\LineSpace}{\LineSpace}}
{
    \pgfsetcolor{\tikz@pattern@color} 
    \pgfsetlinewidth{0.4pt}
    \pgfpathmoveto{\pgfqpoint{0pt}{\LineSpace}}
    \pgfpathlineto{\pgfqpoint{\LineSpace + 0.1pt}{-0.1pt}}
    \pgfusepath{stroke}
}

\makeatother 
\newdimen\LineSpace
\tikzset{
    line space/.code={\LineSpace=#1},
    line space=3pt
}

\newcommand{\ta}{\tilde{\alpha}}
\newcommand{\eps}{\varepsilon}
\newcommand{\Sh}{\mathcal{S}_h}
\newcommand{\uh}{u_h}
\newcommand{\vh}{v_h}
\newcommand{\wh}{w_h}
\newcommand{\zh}{z_h}
\newcommand{\Vhk}{V_h^k}
\newcommand{\Uh}{U_h}

\newcommand{\hmod}[1]{{#1}^{\mathrm{m}}_h}
\newcommand{\himod}[2]{{#1}^{\text{m}}_{#2,h}}

\newcommand{\Uhm}{\hmod{U}}

\newcommand{\Whm}{\hmod{W}}

\newcommand{\Zhmk}{\himod{Z}{\varkappa}}
\newcommand{\uhm}{\hmod{u}}
\newcommand{\vhm}{\hmod{v}}
\newcommand{\whm}{\hmod{w}}
\newcommand{\zhm}{\hmod{z}}
\newcommand{\sihm}{\himod{s}{i}}
\newcommand{\sjhm}{\himod{s}{j}}

\newcommand{\uhp}{u^{\text{pos}}_h}

\newcommand{\gam}{\boldsymbol{\gamma}}
\newcommand{\tgam}{\tilde{\boldsymbol{\gamma}}}
\newcommand{\gh}{\mathbf{g}_h}
\newcommand{\Rg}{\hmod{R}\!(\gam)}
\newcommand{\Rh}{\hmod{R}\!}
\newcommand{\Rtg}{\hmod{R}\!(\tgam)}
\newcommand{\Ihk}{I^k_h}
\newcommand{\Iwhk}{\tilde{I}^k_h}
\newcommand{\Oc}[1]{\Omega_{#1}^\textfrak{c}}

\newcommand{\g}{\nabla}
\newcommand{\embd}{\hookrightarrow}

\newtheorem{rem}{Remark}

\newtheorem{theorem}{Theorem}[section]
\newtheorem{lemma}[theorem]{Lemma}
\newtheorem{corollary}[theorem]{Corollary}

\theoremstyle{definition}

\theoremstyle{remark}

\numberwithin{equation}{section}

\begin{document}

\title[Higher order energy-corrected FEM]{Higher order energy-corrected finite element methods}

\author{Thomas Horger, Petra Pustejovska and Barbara Wohlmuth}
\address{Technische Universit\"{a}t M\"{u}nchen, Boltzmannstra\ss{}e 3, 85748 Garching bei M\"{u}nchen}
\curraddr{}
\email{horger@ma.tum.de,petra.pustejovska@ma.tum.de,wohlmuth@ma.tum.de}
\thanks{The financial support by the German Research Foundation (DFG) trough grants WO 671/11-1, WO 671/13-2  and WO 671/15-1 (within SPP 1748)  is gratefully acknowledged.}

\subjclass[2010]{65N15, 65N30}

\date{}

\dedicatory{}

\keywords{Corner singularities, energy-corrected finite element methods, pollution effect, re-entrant corners, higher order}

\begin{abstract}
	The regularity of the solution of elliptic partial differential equations in a polygonal domain with re-entrant corners is, in general, reduced compared to the one on a smooth convex domain. This results in a best approximation property for standard norms which depend on the re-entrant corner but does not increase with the polynomial degree. Standard Galerkin approximations are moreover affected by a global pollution effect. Even in the far field no optimal error reduction can be observed. Here, we generalize the energy-correction method for higher order finite elements. It is based on a parameter-dependent local modification of the stiffness matrix. We will show firstly that for such modified finite element approximation the pollution effect does not occur and thus optimal order estimates in weighted $L^2$-norms can be obtained. Two different modification techniques are introduced and illustrated numerically. Secondly we propose a simple post-processing step such that even with respect to the standard $L^2$-norm optimal order convergence can be recovered.
\end{abstract}

\maketitle

\section{Introduction}

The standard theory of finite element methods on a family of uniformly refined meshes for elliptic partial differential equations guarantees an optimal order a priori error estimate for the approximation of degree $k$, namely $\mathcal{O}(h^{k})$ in the $H^1$-norm, provided the exact solution is in the Sobolev space $H^{k+1}(\Omega)$. Such regularity assumption on the solution can be violated by many factors, for example by low regularity of the right-hand side of the considered partial differential equation, non-smoothness of the boundary data or non-convexity. For example, for non-convex polygonal domains, the solution looses, in general, its smoothness due to the presence of the corner singularities. One can show then that independently of the approximation order $k$ the error in the $H^1$-norm will exhibit a reduced global rate $h^s$, $0<s<1$, with $s$ dependent of the maximal interior angle of the domain. Such behavior of the approximation scheme is well-known, see \cite{Brenner2008,Ciarlet2002,Grisvard1985,Strang2008} and the references therein. Moreover, the standard Aubin--Nitsche trick does not yield an extra $h$ factor in the $L^2$-norm estimate. Hence, the error decay of the finite element solution is qualitatively slower than the one of the best $L^2$-approximation. This gap between best approximation and Galerkin approximation is also known as pollution effect.

Up to now, there are several approaches which can handle the pollution. One possibility is an approach of adaptive mesh refinement near the re-entrant corner, see, e.g., \cite{Apel1998,Babuska1970,Schatz1979} and recently \cite{Bacuta2005,Bacuta2007,Brenner2011,Chen2010}, or even $h\!-\!p$ adaptivity, combining an increasing polynomial approximation order and a refinement of the mesh towards the singularity, see \cite{Babuska1973,Demkowicz2002,Elliotis2005,Melenk2002}, using least-squares finite element methods in weighted norms, see, e.g., \cite{Jeong2016,Lee2006}, or singular/dual singular function methods, see \cite{Fix1973,Blum1982}.  Alternative, and mainly engineering oriented approaches are the finite element space enrichment method, e.g., \cite{Babuska1973,Blum1982,Strang2008}, an elimination of the dominant error terms by  extrapolation, see \cite{Blum1988,Blum1990a}, or post-processed finite element methods using stress intensity factors, see \cite{Nkemzi2013,Seokchan2016}. Most of these techniques are well understood and of practical interest only in case of low order finite element approximations. 

In this work, we generalize an alternative approach based on the energy-corrected modification of the standard Galerkin approximation, as proposed in \cite{Egger2014} for linear finite elements, and further studied in \cite{Huber2015, John2015,Ruede2014}. Such schemes involve explicit knowledge of the asymptotic expansion of the solution near the singular corner and in comparison to the schemes mentioned above they require only a local and rather trivial change of the standard stiffness matrix. The idea  is similar to the one from \cite{Zegner1978} as proposed for low order finite differences and further generalized in \cite{Ruede1989,Ruede1986}. As a study problem, we consider the Poisson problem with homogeneous Dirichlet boundary conditions on a bounded polygonal domain $\Omega\subset\mathbb{R}^2$ with a single re-entrant corner $\omega\in(\pi,2\pi)$:
\begin{align}\label{prob}
	-\Delta w = g \quad \text{in }\Omega,\qquad w = 0 \quad \text{on } \partial\Omega,
\end{align}
for a given right-hand side $g\in L^2(\Omega)$ and unknown function $w\in H^1_0(\Omega)$ but in this case, in general, $w\not\in H^2(\Omega)$. We use here the standard notation of $L^2(\Omega)=H^0(\Omega)$, $H^m(\Omega)$, $m\in\mathbb{N}_0$ for the Sobolev spaces, and by $\|\cdot \|_{k,\Omega}$ we denote the corresponding norm. For simplicity, if clear from the context, we skip the $\Omega$-symbol in the norm-notation. By $H^1_0(\Omega)$, we denote the subspace of $H^1(\Omega)$ of all functions with vanishing trace. 

For the linear approximations, i.e., $k=1$, for $\omega>\pi$ a reduced best approximation property is to be observed:
\begin{align*}
	\inf_{\vh\in V_h}\Big(\|w-\vh\|_0 + h \|\nabla(w-\vh)\|_0 \Big) \lesssim h^{1+\frac{\pi}{\omega}},
\end{align*}
and one can show that the standard Galerkin approximation $\wh$ is non-optimal with respect to it. It actually satisfies only
\begin{align}\label{standG}
	\|w-\wh\|_0 + h^{\frac{\pi}{\omega}} \|\nabla(w-\wh)\|_0 \lesssim h^{2\frac{\pi}{\omega}}.
\end{align}
Egger et al. in \cite{Egger2014} proposed a construction of a new finite element approximation closing this gap. Moreover, if the error is measured in certain weighted spaces reflecting the singular behavior of the solution $w$, they obtained the standard best approximation, namely:
\begin{align*}
	\inf_{\vh\in V_h}\Big(\|r^{1-\frac{\pi}{\omega}}(w-\vh)\|_0 + h \|r^{1-\frac{\pi}{\omega}} \, \nabla(w-\vh)\|_0 \Big) \lesssim h^2,
\end{align*}
with $r$ being the distance to the re-entrant corner.

The observation of \cref{standG} is even more apparent for higher order methods, i.e., $k>1$, since the higher accuracy is completely lost for the non-smooth solutions. From this point of view,  standard higher order approximations on uniformly refined meshes are not widely used in this context. Our motivation is now to generalize the energy-corrected approach, and show that it is quasi-optimal with respect to the best approximation for the polynomial degree $k$:
\begin{align*}
	\inf_{\vh\in V_h}\Big(\|r^{k-\frac{\pi}{\omega}}(w-\vh)\|_0 + h \|r^{k-\frac{\pi}{\omega}} \nabla(w-\vh)\|_0 \Big) \lesssim h^{k+1}.
\end{align*}
This result will be stated in the main theorem of this work, i.e., in \cref{mainthm}. Also, if one is not interested in the use of the weighted norms $L^2_(\cdot)$, one can obtain the error bound in the standard $L^2$-norm for a post--processed solution $w^{\text{pos}}_h$, i.e.,
\begin{align*}
	\|w-w^{\text{pos}}_h\|_0 + h \| \nabla(w-w^{\text{pos}}_h)\|_0 \lesssim h^{k+1}.
\end{align*}
This will be more elaborated in \cref{col}.

Additionally to the standard Sobolev spaces, we shall use Kondratiev-type weighted spaces, defined for $m\in\mathbb{N}_0$ as follows:
\begin{align*}
	H^m_\beta(\Omega)\coloneqq\left\{f\in L^2_\beta(\Omega): r^{\beta-m+|\mu|} D^{\mu} f \in L^{2}(\Omega), |\mu|\leqslant m \right\},
\end{align*}
where  $\mu$ is a multi-index with $|\mu|=\sum_{i=1}^{2}{\mu_i}$, and $D^\mu$ represents the $\mu$-th generalized derivative, see, e.g., \cite{Kondratjev1967,Kozlov1997}. The space $H^{m}_\beta(\Omega)$ is equipped with the norm 
\begin{align*}
	\|f\|^2_{H^{m}_\beta(\Omega)}=\|f\|^2_{m,\beta}\coloneqq \sum_{|\mu|\leqslant m} \| r^{\beta-m+|\mu|} D^{\mu} f \|^2_0.
\end{align*}
We will also frequently use the following continuous embedding properties of the weighted Sobolev spaces for positive integers $m$ and $l$:
\begin{alignat}{2}
	H^{m+l}_{\alpha+l}(\Omega) \embd H^{m}_{\alpha}(\Omega), \quad (\alpha\in\mathbb{R}), \qquad \text{and} \qquad
	H^{m+l}_{\alpha}(\Omega) \embd \mathcal{C}(\overline{\Omega}), \quad (\alpha<m).\label{embd2}
\end{alignat}
For more details, see, e.g., \cite{kufner_1985}. The $L^2$-inner product and duality product on $\Omega$ are denoted by $(\cdot,\cdot)_\Omega$ and $\langle\cdot,\cdot\rangle_\Omega$. Further, the symbol $\lesssim$ is used for $\leqslant c$, where $c$ is a generic positive constant, always independent of the mesh size $h$.

The rest of this paper is organized as follows: In \cref{S:prob}, we introduce the energy-corrected finite element method and state our main theoretical result, \cref{mainthm}. \Cref{S:prelim} provides some standard auxiliary results. In \cref{S:proof}, we focus on the proof of our main results exploiting re-iteration techniques to obtain the higher order results.  We propose 
in \cref{S:gamma} two explicit options to define a suitable modification satisfying the previously formulated abstract conditions. Finally, in \cref{S:numerics}, we illustrate numerically the theoretical results on three different domains and finite element approximations up to to fourth order.

\section{Problem definition and main theorem}\label{S:prob}
We  approximate problem \cref{prob} by continuous finite elements of polynomial degree $k$. Whereby we denote the finite element space by
\begin{align*}
	\Vhk \coloneqq \big\{\vh\in H^1_0(\Omega),\, \vh|_{T}\in P_k(T)\ \ \forall T\in\mathcal{T}_h \big\},
\end{align*}
with $h>0$ being the mesh size, by $\{\mathcal{T}_h\}_{h>0}$ we denote a family of uniformly refined meshes, and by $P_k(T)$ the $k$-order polynomials on $T$. The discrete energy-corrected formulation of the problem \cref{prob} then takes the form: Find $\whm\in\Vhk$ such that
\begin{align}\label{modific}
	a_h(\whm,\vh) = ( g,\vh)_\Omega \qquad \forall \vh\in\Vhk,
\end{align}
where $a_h(u,v)\coloneqq a(u,v) - c_h(u,v)$ for all $u,v\in H^1_0(\Omega)$. The modification term $c_h(\cdot,\cdot)$ is defined only on a $\mathcal{O}(h)$-surrounding $\Sh$ of the re-entrant corner, see \cref{F:shmesh}.

\begin{figure}[ht!]
	\begin{center}
	\includegraphics[width = 0.3\textwidth]{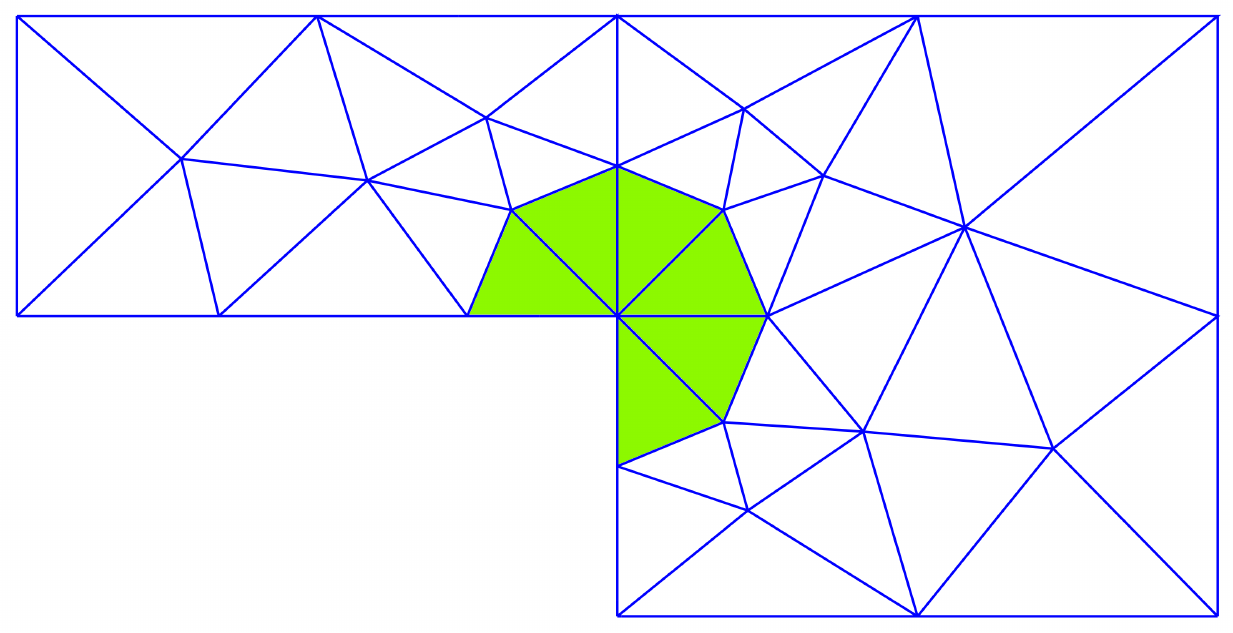}\hfill
	\includegraphics[width = 0.3\textwidth]{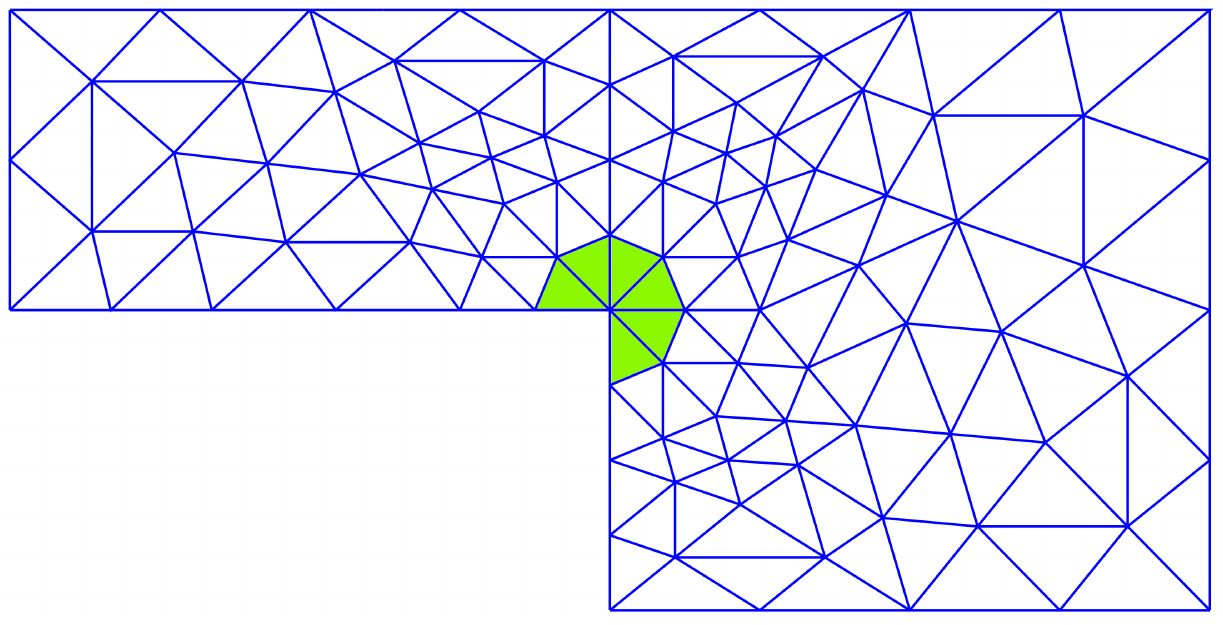}\hfill
	\includegraphics[width = 0.3\textwidth]{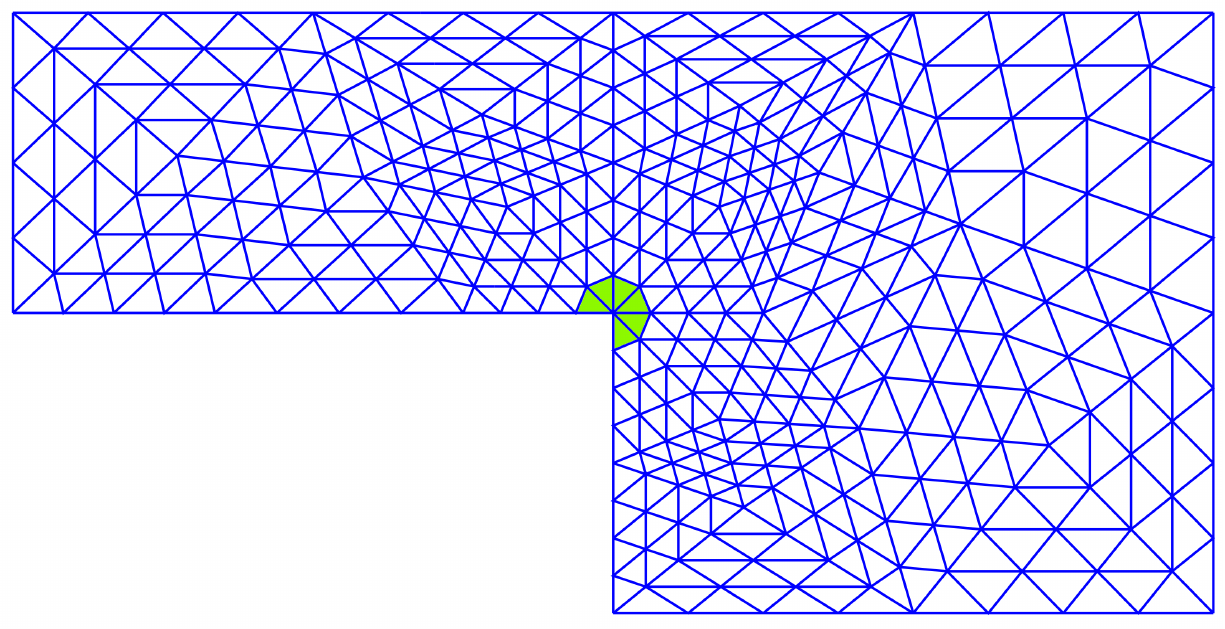}
	\end{center}
	\caption{Graphical illustration of the local uniformness of the initial mesh and two uniform refinement steps. By the green area we depict the $\Sh$-surrounding with the area of  $\mathcal{O}(h^2)$. Note, that in this example $\Sh$ is defined by a single element layer which, in general, does not need to be sufficient.}
	\label{F:shmesh}
\end{figure}

From \cref{modific}, we directly obtain a modified version of the Galerkin orthogonality:
\begin{align}\label{GO}
	a(w-\whm,\vh)+ c_h(\whm,\vh) =0 \qquad \text{for all }\vh\in\Vhk.
\end{align}

At the moment, we require on $c_h(\cdot,\cdot)$ only abstract conditions preserving uniform ellipticity, continuity and symmetry of the bilinear form $a_h(\cdot,\cdot)$. Namely, we assume for the rest of the paper that for all $w,v\in H^1_0(\Omega)$ it holds:
\begin{align*}\label{A}\tag{A}
	a_h(w,w)\gtrsim \|w\|^2_{1}, \qquad |c_h(w,v)|\lesssim \|w\|_{1,\Sh}\|v\|_{1,\Sh}, \qquad c_h(w,v)=c_h(v,w).
\end{align*}
Note that for $c_h(\cdot,\cdot)\equiv 0$, one obtains the standard Galerkin scheme. In that case, we denote the approximate solution by $\wh\in\Vhk$.
One can show that the pollution effect, as discussed above, is contained in the scheme if and only if the pollution function defined as:
\begin{align}\label{gh}
	g_h(w) \coloneqq a(w-\whm,w-\whm) - c_h(\whm,\whm),
\end{align}
converges sub-optimally. More on the pollution effect and its relation to the pollution function can be found the energy-correction method literature \cite{Egger2014,Ruede1989, Ruede2014,Ruede1986,Zegner1978}. Nevertheless, not all parts of the pollution function $g_h(w)$ are of an decreased order. To see this, and thus to specify the sufficient condition on the optimal convergence, we will use an expansion of the solution $w$ into the known eigenfunctions (also called singular functions) of the Laplace operator which can be expressed in the polar coordinates $(r,\theta)$ with respect to the re-entrant corner as:
\begin{align}\label{sidef}
	s_i =\eta(r) r^{\lambda_i}\sin(\lambda_i\theta)=\eta(r) r^{\frac{i\pi}{\omega}}\sin\left(\frac{i\pi}{\omega} \theta\right), \qquad i\in\mathbb{N},
\end{align}
where $\eta(r)$ is a cut-off function on $\mathcal{B}_{a}^2(0)=\{x\in\mathbb{R}^2:\ \|x\|_{l^2}<a\}$ for sufficiently small $a$  such that $\text{supp}(\eta)\subset \Omega$, and for $0<b<a$ it holds: $\eta(r)=1$ for $0\leqslant r\leqslant b$, $\eta(r)=0$ for $r\geqslant a$, and $\eta$ is sufficiently smooth on $[b,a]$. The factors $\lambda_i=\frac{i\pi}{\omega}$ are the corresponding $\omega$-dependent eigenvalues, uniquely ordered as: $\lambda_1>\lambda_2>...$, for all $\omega\in(\pi,2\pi)$. Note that $s_i\in H^{2}_{1-\lambda_i+\eps}(\Omega)$ and also $s_i\in H^{k+1}_{k-\lambda_i+\eps}(\Omega)$, $\eps>0$, $\Delta s_i =0$ on $\Omega\cap\mathcal{B}_b^2(0)$, $i\in\mathbb{N}$, and $s_i$ are mutually orthogonal, i.e., $(s_i,s_j)_\Omega=0$ for $i\neq j$. More on the singular functions and the eigenvalue problem can be found, e.g., in \cite{Dauge1988,Kozlov1997}. The general solution of \cref{prob} admits then for a regular right-hand side an expansion, see \cite{Kondratjev1967}, which can be formulated as follows:
\begin{lemma}[Expansion of the solution]\label{l:decomp}
	Let $\Omega$ be a polygonal with a single re-entrant corner of an angle $\omega\in(\pi,2\pi)$, $m\in\mathbb{N}_0$, and $g\in H^m_{-\gamma}(\Omega)$ for $\gamma>1+m-\lambda_1$. Also, let $N_w\in\mathbb{N}$ be such that 
	\begin{align*}
		\forall  i=1,...,N_w: \quad \lambda_i<1+m+\gamma, \qquad \text{and} \qquad \lambda_{N_w+1}>1+m+\gamma.
	\end{align*}
	Then the unique solution $w\in H^1_0(\Omega)$ of \cref{prob} belongs also to $H^{m+2}_\gamma(\Omega)$ and it can be decomposed into
	\begin{align}\label{sumdec}
		w=\sum_{i=1}^{N_w} \mu^w_i s_i + W, 
	\end{align}
	where $W\in H^{m+2}_{-\gamma}(\Omega)\cap H^1_0(\Omega)$ and $\mu_i^w\in\mathbb{R}$. Moreover, it satisfies the following a~priori estimates:
	\begin{align}\label{aest}
		\|w\|_{m+2,\gamma}\lesssim \|g\|_{m,-\gamma}, \qquad \sum_{i=1}^{N_w} |\mu^w_i| + \|W\|_{m+2,-\gamma} \lesssim \|g\|_{m,-\gamma}.
	\end{align}
\end{lemma}
The function $W$ is called the smooth remainder, and $\mu^w_i$ are called the stress intensity factors. Realizing that the singular functions $s_i$ and the smooth remainder $W$ are also solutions of Poisson problems with specific right-hand sides, one can split the approximation error by:
\begin{align}\label{errsplit}
	w-\whm = \sum_{i=1}^{N_w} \mu^w_i(s_i - \sihm) + W-\Whm,
\end{align}
where $\sihm$ and $\Whm$ represent the modified approximations of $s_i$ and $W$, respectively.

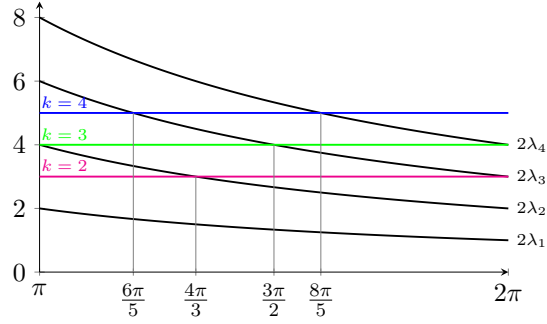
\begin{figure}
\begin{center}
\begin{tikzpicture}[scale=0.9,font=\large]
	\def\piv{3.1415926535897932385}
	\def\dpiv{6.283185307179586477}
	\def\ca{3.769911184}
	\def\cb{4.188790205}
	\def\cc{4.71238898}
	\def\cd{5.026548246}
	\begin{axis}[
		yscale=0.7,ymax=8.5,ymin=0,xmin=pi,xmax=2*pi,xtick={\piv,\ca,\cb,\cc,\cd,\dpiv},
		xticklabels={$\pi$,$\frac{6\pi}{5}$,$\frac{4\pi}{3}$,$\frac{3\pi}{2}$,$\frac{8\pi}{5}$,$2\pi$},axis x line=bottom,axis y line=left, ylabel={},
		every axis x label/.style={at={(ticklabel* cs:1.05)},anchor= north east},clip=false
		]
	\addplot[domain=pi:2*pi,color=black,thick,samples=100] {2*pi/x};
	\addplot[domain=pi:2*pi,color=black,thick,samples=100] {4*pi/x};
	\addplot[domain=pi:2*pi,color=black,thick,samples=100] {6*pi/x};
	\addplot[domain=pi:2*pi,color=black,thick,samples=100] {8*pi/x};
	\addplot[domain=pi:2*pi,color=magenta,thick] {3};\addplot[domain=pi:2*pi,color=green,thick] {4};\addplot[domain=pi:2*pi,color=blue,thick] {5};
	\draw[thin,color=gray] (axis cs:\ca,0) -- (axis cs:\ca,5);\draw[thin,color=gray] (axis cs:\cb,0) -- (axis cs:\cb,3);
	\draw[thin,color=gray] (axis cs:\cc,0) -- (axis cs:\cc,4);\draw[thin,color=gray] (axis cs:\cd,0) -- (axis cs:\cd,5);

	\node [anchor=west] at (axis cs:\dpiv,1) {\scriptsize{$2\lambda_1$}};
	\node [anchor=west] at (axis cs:\dpiv,2) {\scriptsize{$2\lambda_2$}};
	\node [anchor=west] at (axis cs:\dpiv,3) {\scriptsize{$2\lambda_3$}};
	\node [anchor=west] at (axis cs:\dpiv,4) {\scriptsize{$2\lambda_4$}};
	\node [anchor=south west,magenta] at (axis cs:\piv-0.045,3-0.1) {\scriptsize{$k=2$}};
	\node [anchor=south west,green] at (axis cs:\piv-0.045,4-0.1) {\scriptsize{$k=3$}};
	\node [anchor=south west,blue] at (axis cs:\piv-0.045,5-0.1) {\scriptsize{$k=4$}};
	\end{axis}
\end{tikzpicture}
\end{center}
	\caption{Graphical illustration of the total number $K$ of essential pollution functions $g_h(s_i)$. Namely, the value $K$ is determined by the number of thick lines below the horizontal approximation line. For example, in the case of the approximation order $k=4$, we have: $K=2$ for $\omega\in(\pi,\frac{6}{5}\pi]$, $K=3$ for $\omega\in(\frac{6}{5}\pi,\frac{8}{5}\pi]$, and $K=4$ for $\omega\in(\frac{8}{5}\pi,2\pi)$. The colored lines represent $k+1$ bounds.}
	\label{F:dist}
\end{figure}

In the theoretical part of this work, we will often use the decomposition \cref{sumdec} and its properties for both, the primal and the dual problem, each of them having different regularities and thus different decomposition series. Therefore, for better readability, we strictly use a notation of $(u,\uh,\uhm)$ and $(z,\zh,\zhm)$ for the primal solution, its standard and modified approximation, and the dual solution, and its standard and modified approximation, respectively. The auxiliary functions, on the other hand, we denote by $w$ or $v$, and correspondingly their finite element approximations by $\wh$ and $\vh$. Note also that the singular functions are the same for both solutions, $u$ and $z$, but we will distinguish between the constants, $\mu^u_i$ and $\mu^z_i$, and the regular remainders $U$ and $Z$. 

\begin{rem}\label{indices}
	For better readability, we specify typical indices and weights for the primal and the dual problem as used in \cref{l:decomp}. Let $k$ be the finite element approximation order, then
	\begin{alignat*}{3}
		\bullet&\text{ primal problem }u:&&&&\\[-0.5em]
		&\quad m=k-1, & \gamma&=\ta\coloneqq 1-\lambda_1+\eps, \qquad  & N_u&=\smash[b]{\left\lfloor(k+1)\frac{\omega}{\pi} -1 \right\rfloor},\\
		\bullet&\text{ dual problem }z:&&&&\\[-0.5em]
		&\quad m=0, & \gamma&=\alpha\coloneqq \ta + k -1 , \qquad & N_z&= N_u,
	\end{alignat*}
	
	\noindent where $\lfloor x \rfloor$, $x\in\mathbb{R}$, represent the floor function, namely $\lfloor x \rfloor=\max(l)$, $l\in\mathbb{N}$,  $l\leqslant x$. 
\end{rem}

Associated with the order $k$ of the finite element space and the angle $\omega$ of the re-entrant corner, we define a correction order $K$: 
\begin{align}\label{Kdef}
	K \coloneqq \left\{  
	\begin{array}{cl}
	\left \lfloor (k+1)\frac{\omega}{2\pi} \right\rfloor , & \quad \text{if }  (k+1)\frac{\omega}{2\pi} \not\in \mathbb{N}, \\[0.3cm]
	(k+1)\frac{\omega}{2\pi} -1 & \quad \text{else.}  
	\end{array} \right.
\end{align}
Then in terms of \cref{sumdec}, the  pollution function $g_h(u)$ can be decomposed into:
\begin{align}\label{ghu}
	g_h(u) = \sum_{i=1}^{K} (\mu_i^u)^2 g_h(s_i) + \sum_{i\neq j=1}^{N_u} \mu_i^u\mu_j^u\, \tilde{g}_h(s_i,s_j) + \text{extra terms},
\end{align}
where we set
\begin{align*}
	\tilde{g}_h(v,w) \coloneqq a(v-\hmod{v},w-\hmod{w}) - c_h(\hmod{v},\hmod{w}), \quad \text{i.e.,}\quad g_h(w) = \tilde{g}_h(w,w).
\end{align*}

As we will show later, $g_h(s_i)$ are essential pollution terms due to their non-conditional reduced convergence order. The possible pollution of $\tilde{g}_h(s_i,s_j)$ is non-essential in a sense that it can be balanced by some local symmetry properties of the mesh. The extra terms do not contribute to the pollution. We aim thus to construct the modification $c_h(\cdot,\cdot)$ in such a way, that we correct the essential pollution terms in~\cref{ghu}. 

\begin{figure}
	\centerline{\includegraphics[width = 0.7\textwidth]{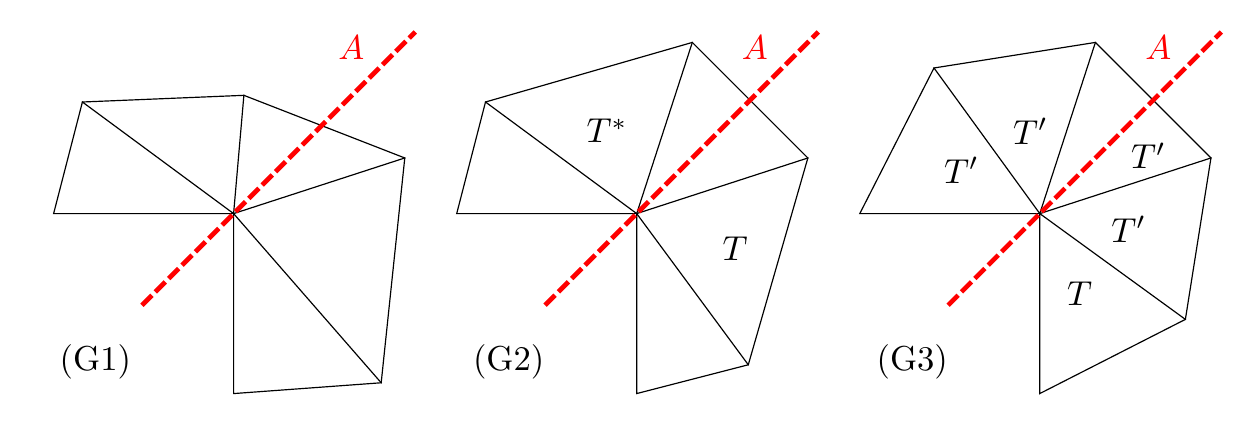}}
	\caption{Symmetry conditions of the local mesh at the re-entrant corner. For simplicity, we depict only a first layer of elements. The dashed line $A$ represents the bisectional axis of the angle $\omega$, element $T^*$ represents the mirror image  across $A$ of $T$, and $T'$ represent the rotation of $T$ by $\frac{\omega}{n}$ around the re-entrant corner. Here $n$ is the number of elements directly attached to it.}
	\label{F:sym}
\end{figure}

Let us now turn our attention to the additional assumptions we lay on the mesh. In the following, we distinguish between several cases of combinations between $\omega$ and the approximation order $k$ and assign to them three different mesh properties (G1)--(G3) which are schematically represented in \cref{F:sym} for a typical example $\omega=3/2\pi$. Note, that (G1)--(G3) are assumed on the $\Sh$-domain of the coarsest refinement level and thus, by uniform refinement, they will automatically hold on each $\Sh$. More precisely, (G1) represents a general mesh without any specific symmetry properties, (G2) stands for a mesh with a local mirror-symmetry, i.e., to each element $T\in\Sh$ exists another $T^*\in\Sh$ which is a mirror image of $T$ across the bisectional axis $A$. The last property (G3) represents a full local symmetry at the singular corner in the sense that $\Sh$ is composed of one base element which is then rotated around the re-entrant corner. We say that an initial mesh has the \cref{U}-property which depends on $k$ and $\omega$ if
\begin{align}\label{U}\tag{U}
	\begin{cases}
		\ \text{(G1) holds for } k=1\ \text{ and }\ \omega<\frac{3}{2}\pi, & \\ 
		\ \text{(G2) holds for } k=1\ \text{ and }\ \omega\geqslant\frac{3}{2}\pi, \text{ or } k=2 \text{ and } \omega<\frac{4}{3}\pi,& \\
		\ \text{(G3) holds for } k=2\ \text{ and }\ \omega\geqslant\frac{4}{3}\pi, \text{ or } k\geqslant 3.&
	\end{cases}
\end{align}
Associated with the \cref{U}-property of a mesh, we make here an assumption on the interpolation property. If the mesh satisfies the \cref{U}-property, we assume that for $\lambda_i + \lambda_j < k+1$ we find 
\begin{align}\label{assint}
	 \int_{\Upsilon} \nabla s_j\cdot \nabla \Ihk s_i\, dx=\int_{\Upsilon} \nabla s_j \cdot \nabla s_i \, dx=0, \qquad i\neq j,
\end{align}
where $\Upsilon$ is the domain $\Sh$ of the coarsest refinement level. We call this assumption \cref{U}-assumption. We point out that \cref{assint} does not hold for arbitrary shaped meshes at the re-entrant corner. However, all our numerical study tests have shown that the \cref{U}-property is sufficient for \cref{assint} to hold, see \cref{tab:u3ortho} in \cref{S:numerics}.

We will show later that under the \cref{U}-assumption the mixed terms $\tilde{g}_h(s_i,s_j)$, $i\neq j=1,...,N_u$, converge optimally, and thus, the global pollution in the error originates strictly from $g_h(s_i)$ terms, $i\leqslant K$. Hence, we call $K$ the correction order, since it represents the number of conditions we lay on the modification. Its value can be simply read out from \cref{F:dist}. Now, we are ready to formulate the main theorem on the order of convergence, the proof will be shifted to the \cref{S:proof}.
\begin{theorem}[Main theorem]\label{mainthm}
	Let $k\in\mathbb{N}$ denotes the polynomial degree of the finite element space, $f \in H^{k-1}_{-\ta}(\Omega)$ for some $  \ta>1-\lambda_1$, and let $\omega$ be the angle of the re-entrant corner. Let us assume that \cref{A} and \cref{U}-assumptions hold and moreover that the modification $c_h(\cdot, \cdot)$ satisfies:
	\begin{align}\label{assmthm}
		g_h(s_i) = \mathcal{O}(h^{k+1}) \qquad \text{for all } i\leqslant K.
	\end{align}
	Then, we have the following optimal convergence results of the modified approximation in the weighted norms with $\alpha = \ta + k-1$:
	\begin{align}\label{woc}
		\|u-\uhm\|_{0,\alpha} \lesssim h^{k+1} \| f\|_{k-1,-\ta} \quad \mbox{and} \quad \|\nabla(u-\uhm)\|_{0,\alpha}\lesssim h^k \|f\|_{k-1,-\ta}.
	\end{align}
\end{theorem}

The fundamental result of \cref{mainthm} is that an error decay of order $k+1$ for the  modified Galerkin approximation can be observed in a suitable norm. However, to obtain $\mathcal{O}(h^{k+1})$ estimates, the weaker weighted $L^2$-norm $\|\cdot\|_{0,\alpha}$ have to be considered. To recover the approximation order with respect to the standard $L^2$-norm, we can introduce a post-process step on $\uhm$.  First note, that the stress intensity factors, see, e.g., \cite{Grisvard1985}, can be also represented by the singular functions $s_{-i}$ of the adjoint operator, i.e.,
\begin{align*}
	\mu^u_i =\frac{1}{i\pi}\int_{\Omega} \left( f s_{-i} + u\Delta s_{-i}\right) dx, \quad\text{where}\quad s_{-i} = \eta(r) r^{-\lambda_i} \sin{\lambda_i \theta}.
\end{align*}
For this characterization, we can define an approximation of the stress intensity factor as: 
\begin{align*}
	\mu_i^h \coloneqq \frac{1}{i\pi}\int_{\Omega} \left( f s_{-i} + \uhm\Delta s_{-i}\right) dx,
\end{align*}
and thus, by \cref{woc} and $\Delta s_{-i} = 0$ on $\Omega\cap\mathcal{B}_a^2(0)$, see the properties after \Cref{sidef}, we obtain directly for all $i\leqslant N_u$ that 
\begin{align}\label{convmu}
	|\mu^u_i-\mu_i^h|=\frac{1}{i\pi}\left|\int_{\Omega} \left( u- \uhm\right)\Delta s_{-i} dx\right|\lesssim \|u-\uhm\|_{0,\alpha}\|\Delta s_{-i}\|_{0,-\alpha}
	\lesssim h^{k+1}.
\end{align}

\begin{corollary}\label{col}
	Let assume that the assumptions of \cref{mainthm} hold. Then,  the post-processed solution defined as
	\begin{align*}
		\uhp \coloneqq \uhm + \sum_{i=1}^{K}\mu_i^h (s_i - \sihm),
	\end{align*}
	converges in standard $L^2$-norm as
	\begin{align*}
		\|u-\uhp\|_{0} \lesssim h^{k+1}.
	\end{align*}
\end{corollary}
Such is a direct consequence of \cref{errsplit} for $u$, the triangle inequality and \cref{convmu}.

\section{Some preliminary results}\label{S:prelim}

The proof of the main theorem strongly relies on the decomposition property \cref{sumdec}. Thus, before we turn our attention to the proof itself, we state here the approximation properties of the smooth reminder $U$ from the expansion \cref{sumdec} with indices from \cref{indices}, i.e., for $U\in H^{k+1}_{-\ta}(\Omega)$.

Note also, that these properties are derived by standard techniques, see for example \cite{Egger2014,John2015} concerning energy-corrections, nevertheless, for the sake of completeness and the readability of the later sections, we include them here. The generalized interpolation errors, as well as approximation estimates of the modified scheme in standard norms, are included in the Appendix. 

\begin{lemma}\label{l:srem}
	Let $U\in H^{k+1}_{-\ta}(\Omega) \cap H^1_0(\Omega)$ with $k\geqslant 1$ and $1-\lambda_1<\ta<1$. Additionally, let \cref{A} holds. Then, for $l=\{0,1\}$, we have the following approximation estimates for the smooth remainder:
	\begin{align}\label{srem}
		\quad\|\nabla^l (U-\Uh)\|_{0,-\ta} \lesssim h^{k+1-l} \,\|U\|_{k+1,-\ta},
	\end{align}
	and
	\begin{align}\label{srem2}
		\|\nabla^l (U-\Uhm)\|_{0,-\ta} \lesssim h^{k+1-l}\, \|U\|_{k+1,-\ta}.
	\end{align}
\end{lemma}

\begin{proof}
	Let us prove  the estimates for the standard Galerkin approximation \cref{srem} first, and then, we will focus on the estimates for the modified scheme \cref{srem2}.\\
	{\it Standard approximation error}

	In what follows, we shall strongly rely on the approximation error properties in the mesh-dependent norms, as derived for the linear approximation in \cite{Blum1988}. More precisely, let us define $\varrho \coloneqq (r^2+\theta h^2)^{1/2}$ for some sufficiently large but fixed $\theta>0$. Then, analogously, for any $1-\lambda_1<\ta<1$ there holds:
	\begin{align}\label{meshdepw}
		\|\varrho^{-\ta}\nabla(U-\Uh)\|_0 
		&\lesssim \|\varrho^{-\ta}\nabla(U-\Ihk U)\|_0 + h^{-1} \|\varrho^{-\ta}(U-\Ihk U)\|_0.
	\end{align}

	First, we derive the estimates for the gradient. By the triangle inequality, the equivalence of the norms $\|\cdot\|_{-\ta}$ and $\|\varrho^{-\ta}\cdot\|_{0}$ on finite element spaces, using again triangle inequality, \cref{meshdepw}, $\varrho^{-\ta}<r^{-\ta}$, and the interpolation estimates \cref{interr}, we have that
	\begin{align*}
		\|\nabla(U&-\Uh)\|_{0,-\ta}\lesssim \|\nabla(U-\Ihk U)\|_{0,-\ta} + \|\varrho^{-\ta}\nabla(\Uh-\Ihk U)\|_{0}\\
		&\lesssim \|\nabla(U-\Ihk U)\|_{0,-\ta} + \|\varrho^{-\ta} \nabla(U-\Ihk U)\|_0 + h^{-1} \|\varrho^{-\ta}(U-\Ihk U)\|_0\\
		&\lesssim \|\nabla(U-\Ihk U)\|_{0,-\ta} + h^{-1} \|U-\Ihk U\|_{0,-\ta}
		\lesssim h^{k} \|U\|_{k+1,-\ta}.
	\end{align*}
	To derive the $L^2_{-\ta}-$bound, we consider the following adjoint problem
	\begin{align}\label{adjpr1}
		-\Delta v = r^{-2\ta} (U-\Uh) \quad \text{in }\Omega, \qquad \text{and} \qquad v = 0 \quad \text{in }\partial \Omega.
	\end{align}
	The right hand side is in $L^2_{\ta}(\Omega)$ since, using \cref{embd2}, $U \in L^{2}_{-\ta}(\Omega)$ and also $\Uh\in L^2_{-\ta}$ for $-\ta>-1$. Thus, the solution of the dual problem admits the regularity result from \cref{l:decomp}: 
	\begin{align}\label{dp1}
		\|v\|_{2,\ta} \lesssim \|r^{-2 \ta}(U-\Uh)\|_{0,\ta}=\|U-\Uhm\|_{0,-\ta}.
	\end{align}
	This yields for the error, together with the integration by parts, standard Galerkin orthogonality, the H\"older inequality, the result for the gradients and the interpolation estimates from \cref{interr}:
	\begin{align*}
		\|U-\Uh\|_{0,-\ta}^2 &= (U-\Uh,r^{-2\ta} (U-\Uh)) = (U-\Uh,-\Delta v)
		=a(U-\Uh,v) \\ &= a(U-\Uh,v-I_h^1 v) \leqslant \|\nabla(U-\Uh)\|_{0,-\ta} \|\nabla(v-I_h^1 v)\|_{0,\ta}\\
		&\lesssim h^{k} \|U\|_{k+1,-\ta}\, h \|v\|_{2,\ta} \lesssim h^{k+1} \|U\|_{k+1,-\ta} \|U-\Uh\|_{0,-\ta}.
	\end{align*}
	Dividing by $\|U-\Uh\|_{0,\ta}$ we can then determine the final $L^2_{-\ta}-$estimate. Let us now turn our attention to the estimates of \cref{srem2} for the modified Galerkin scheme.\\
	{\it Step 2: Modified approximation error}

	The $L^2_{-\ta}$ estimate can be proven by a similar duality argument as above, namely for \cref{adjpr1} with right hand-side $r^{-2\ta} (U-\Uhm)$. Analogously, we obtain for the solution $v$ that 
	\begin{align}\label{dp}
		\|v\|_{2,\ta} \lesssim \|U-\Uhm\|_{0, -\ta}.
	\end{align}
	We denote by $\vhm\in\Vhk$ the $k-$th order modified Galerkin approximation of $v$. Hence, using \cref{GO} for $\vhm$, $\Uhm$, and $\Uh$, the triangle and Cauchy--Schwarz inequality, we derive 
	\begin{align*}
		\|U-\Uhm\|_{0,-\ta}^2 &= (U-\Uhm,r^{-2\ta}(U-\Uhm))=(U-\Uhm,-\Delta v)=a(U-\Uhm,v)\\
		&=a(U-\Uhm,v-\vhm)-c_h(\Uhm,\vhm)=a(U-\Uh,v)-c_h(\Uh,\vhm)\\
		&\lesssim \|U-\Uh\|_{0,-\ta} \|\Delta v \|_{0,\ta} + \|\nabla \Uh\|_{L^2(\Sh)}\|\nabla \vhm\|_{L^2(\Sh)}.
	\end{align*}
	The first term can be estimated by the help of the standard approximation estimate from \cref{srem}, $\|\Delta v \|_{0,\ta}\lesssim \|v\|_{2,\ta}$ and \cref{dp}, thus
	\begin{align*}
		\|U-\Uh\|_{0,-\ta}\|\Delta v \|_{0,\ta} \lesssim h^{k+1}\|U\|_{k+1,-\ta} \|U-\Uhm\|_{0,-\ta}. 
	\end{align*}
In the rest of the proof, we use results from the Appendix.
	For the $\Sh-$supported terms we use the triangle inequality, similar scaling arguments as in \cref{Shest}, and again the first part of \cref{srem}:
	\begin{align*}
		\|\nabla\Uh\|_{L^2(\Sh)} &\leqslant \|\nabla(U-\Uh)\|_{L^2(\Sh)} + \|\nabla U\|_{L^2(\Sh)}\\
		&\lesssim h^{\ta} \left(\|\nabla(U-\Uh)\|_{L^2_{-\ta}(\Sh)} + h^k \|U\|_{k+1,-\ta}\right)
		\lesssim h^{\ta+k} \|U\|_{k+1,-\ta},
	\end{align*}
	and similarly, by \cref{standest2} together with the regularity result \cref{dp} for $v$ as the solution of the Poisson problem:
	\begin{align*}
		\|\nabla \vhm\|_{L^2(\Sh)} &\lesssim \|v-\vhm\|_{H^1(\Sh)}+\|\nabla v\|_{L^2(\Sh)}
		\lesssim h^{1-\ta}\|v\|_{2,\ta} \lesssim h^{1-\ta} \|U-\Uhm\|_{0,-\ta}.
	\end{align*}
	Summing all the estimated terms above and dividing the sum by $\|U-\Uhm\|_{0,\ta}$, we obtain the $L^2_{-\ta}-$approximation error. It remains to show the bound for the gradient:
	\begin{align*}
		\|\nabla(U&-\Uhm)\|_{0, -\ta}\leqslant \|\nabla(U-\Ihk U)\|_{0,-\ta}+\|\nabla(\Ihk U-\Uhm)\|_{0,-\ta}\\
		&\lesssim h^k\|U\|_{k+1,-\ta} + h^{-1} (\|\Ihk U-U\|_{0,-\ta}+\|U-\Uhm\|_{0,-\ta})\lesssim h^k\|U\|_{k+1,-\ta},
	\end{align*}
	where we used the triangle inequality, the interpolation estimates (\cref{l:intest}), the generalized inverse inequality (\cref{l:inv}), and the above derived $L^2_{-\ta}-$estimate \cref{srem}.
\end{proof}

\section{Proof of the main theorem}\label{S:proof}

Before we start, let us define a restricted interpolation $\Iwhk\in\Vhk$ such that for each $v\in\mathcal{C}(\overline{\Omega})$ holds: $\Iwhk v =0$ on $\overline{\Sh}$ and $\Iwhk v (\mathbf{x}) = v(\mathbf{x})$ for all interpolation nodes $\mathbf{x}\not\in\overline{\Sh}$. Then, for each $\wh\in\Vhk$ it holds: $c_h(\Iwhk v, \wh)=0$. Moreover, $\Iwhk v \neq \Ihk v $ only on a $\mathcal{O}(h)$-surrounding of the singular point. Its approximation properties are then stated in the Appendix, see \cref{Iwhk}.

One of the key ingredients of the proof below will be generalized Wahlbin results. Namely, we will use a standard energy estimate:
\begin{align}\label{whlb1}
	\|\nabla(u-\uhm)\|_{0,A} \lesssim \|\nabla(u - \Ihk u)\|_{0,B} + \|u-\uhm\|_{0,B},
\end{align}
for $A\subset B$ such that $\mathbf{x}_s\not\in \overline{B}$ and $\text{dist}(\partial A\setminus\partial \Omega, \partial B\setminus\partial \Omega)>0$, see \cite[Chapter III., Theorem 9.1]{Wahlbin1991} or \cite{Nitsche1974}. 
Additionally, we will use a generalized result in the weighted norms: 
\begin{align}\label{whlb}
	\|r^{\beta}\nabla(u-\uhm)\|_{0} &\lesssim \|r^{\beta}\nabla(u-\Iwhk u)\|_{0} + \|r^{\beta-1}(u-\uhm)\|_{0},
\end{align}
for $\beta\geqslant1$. This estimate can be obtained from the local energy estimates provided by \cite[Chapter III., Corollary 9.1.]{Wahlbin1991} for standard norms on subdomains. As it is done in \cite{Melenk2012}, a decomposition of $\Omega$ into overlapping dyadic balls then gives the estimate in weighted norms for standard conforming finite elements. However, the result holds also true for the modified discrete solution since the modified Galerkin orthogonality \cref{GO} for $\Iwhk u$ reduces to the standard one. 

As we have already mentioned, we shall base the proof of \cref{mainthm} on the decomposition in \cref{sumdec}. In contrast to \cite{Egger2014} which is restricted to  linear finite elements, we have to deal also with the pollution in the $H^1$-norm. 
To obtain the required $k+1$ power of $h$ in \cref{woc}, we carry out the proof recursively. In the first step, see \cref{weiconv}, we use stronger assumptions as they are formulated in the main theorem.  In the rest of the section, these assumptions will be weakened.

\begin{lemma}\label{weiconv}
	Assume that for some $\sigma\leqslant k$ it holds:
	\begin{align}\label{assums}
		a(s_i-\sihm,s_j-\sjhm) - c_h(\sihm,\sjhm) = \mathcal{O}(h^{\sigma+1}),
	\end{align}
	$i,j\leqslant N_u=N_z$, for $N_u$ as defined in \cref{indices}. Then, for $f \in H^{k-1}_{-\ta}(\Omega)$ and  $1-\lambda_1 < \ta < 1$, we have
	\begin{align}\label{errbd}
		\|u-\uhm\|_{0,\alpha} \lesssim h^{\sigma+1} \|f\|_{k-1,-\ta} \quad \mbox{and} \quad \|\nabla(u-\uhm)\|_{0,\alpha}\lesssim h^\sigma \|f\|_{k-1,-\ta}.
	\end{align}
\end{lemma}
\begin{proof}
	We will show the required error bounds \cref{errbd} by induction with increasing the weights from $\ta$ to $\ta+\sigma-1\leqslant \alpha$. For that we define the induction variables $\varkappa=1,...,\sigma$, and $\beta(\varkappa)\coloneqq \ta+\varkappa-1=\varkappa-\lambda_1+\eps$. 
	
	{\it Initial step: $\varkappa=1$}
	The low order results from \cite{Egger2014} directly apply also for higher order finite elements, and thus we have directly $\|\g^l(u-\uhm)\|_{0,\ta}\lesssim h^{2-l}\|f\|_{0,-\ta} \lesssim h^{2-l}\|f\|_{k-1,-\ta}$, $l=0,1$.

	{\it Induction step: $2 \leqslant \varkappa\leqslant\sigma$}:	
	We show the error bounds for the general $\varkappa$-step using the result from the previous $(\varkappa\!-\!1)$-step, i.e.,
	\begin{align}\label{ia}
		 \|u-\uhm\|_{0,\beta(\varkappa-1)} \lesssim h^{\varkappa}\|f\|_{k-1,-\ta}, \quad  \|\g(u-\uhm)\|_{0,\beta(\varkappa-1)} \lesssim h^{\varkappa-1}\|f\|_{k-1,-\ta}.
	\end{align}

	By \cref{whlb}, the interpolation error \cref{interr} for $\Iwhk$ (see \cref{Iwhk}) and \cref{ia} from the previous induction step, we can directly obtain the $H^1_{\beta(\varkappa)}-$estimate:
	\begin{align}\label{gerr}
		\|\nabla(u-\uhm)\|_{0,\beta(\varkappa)} &\lesssim \|\nabla(u-\Iwhk u)\|_{0,\beta(\varkappa)} + \|u-\uhm\|_{0,\beta(\varkappa-1)}
		\lesssim h^{\varkappa}\|f\|_{k-1,-\ta}.
	\end{align}

	Let us continue with the proof of the $L^2_{\beta(\varkappa)}-$estimate. We consider the dual problem 
	\begin{align}\label{refdp}
		-\Delta z = r^{2\beta(\varkappa)} (u-\uhm)\quad \text{in}~\Omega, \qquad \text{and} \qquad z = 0 \quad \text{on}~\partial \Omega.
	\end{align}
	The right hand side $r^{2\beta(\varkappa)}(u-\uhm)$ is in $L^2_{-\beta(\varkappa)}(\Omega)$ since $u \in H^{k+1}_{\alpha}(\Omega)\subset L^2_{\beta(\varkappa)}(\Omega)$. Thus, by \cref{l:decomp} we can expand the solution $u$ of the original Poisson problem and the solution $z$ of the dual problem as
	\begin{align*}
		u = \sum \limits_{i=1}^{N_u} \mu^u_i s_i + U, \quad \text{and} \quad z =\sum \limits_{j=1}^{\mathclap{N_z(\varkappa)}} \mu^z_j s_j + Z_{\varkappa},
	\end{align*}
	with $U \in H^{k+1}_{-\ta}(\Omega)$ and $Z_{\varkappa}\in H^2_{-\beta(\varkappa)}(\Omega)$ with $N_z(\varkappa)\leqslant N_z$. Further, the following a~priori estimates hold
	\begin{align}\label{aed}
		\sum \limits_{i=1}^{N_u} |\mu^u_i| + \|U\|_{k+1,-\ta} \lesssim \|f\|_{k-1,-\ta}, \ \sum \limits_{j=1}^{\mathclap{N_z(\varkappa)}} |\mu^z_j| + \|Z_{\varkappa}\|_{2,-\beta(\varkappa)} \lesssim \|u-\uhm\|_{0,\beta(\varkappa)}.
	\end{align}
	We can then rewrite the $L^2_{\beta(\varkappa)}$-error by \cref{refdp}, integration by parts, modified Galerkin orthogonality \cref{GO}, and the expansions of $u$ and $z$ as:
	\begin{equation}\label{derr}
	\begin{aligned}
		\|u-\uhm\|_{0,\beta(\varkappa)}^2 &= (u-\uhm,r^{2\beta(\varkappa)}(u-\uhm))= a(u-\uhm,z-\zhm) - c_h(\uhm,\zhm) \\
		&= \sum_{i=1}^{N_u}\, \sum_{j=1}^{\mathclap{N_z(\varkappa)}} \mu^u_i \, \mu^z_j \big[ a(s_i-\sihm,s_j-\sjhm) - c_h(\sihm,\sjhm)\big]\\
		&\quad + \sum \limits_{j=1}^{\mathclap{N_z(\varkappa)}} \mu^z_j \big[a(U-\Uhm,s_j-\sjhm) - c_h(\Uhm,\sjhm)\big]\\
		&\quad + \sum \limits_{i=1}^{N_u} \mu^u_i \big[ a(s_i-\sihm,Z_{\varkappa}-\Zhmk) - c_h(\sihm,\Zhmk)\big]\\
		&\quad + a(U-\Uhm,Z_{\varkappa}-\Zhmk) - c_h(\Uhm,\Zhmk).
	\end{aligned}
	\end{equation}
	In the following, we consider each term of \cref{derr} individually. We can estimate the first double-sum by the assumption of the lemma and by the a~priori estimates of the expansions \cref{aed}, i.e.,
	\begin{align}\label{term1}
		\mu^u_i \mu^z_j \big[ a(s_i-\sihm,s_j-\sjhm)-c_h(\sihm,\sjhm)\big] \lesssim h^{\sigma+1}\|f\|_{k-1,-\ta} \|u-\uhm\|_{0,\beta(\varkappa)}.
	\end{align}
	For the second term, we apply the modified Galerkin orthogonality \cref{GO}, integration by parts, Cauchy Schwarz inequality, the fact that $\|\Delta s_j\|_{0,\ta}\lesssim \|s_j\|_{2,\ta}\leqslant c$ and the estimates of \cref{l:srem} as well as the regularity estimates 
	\begin{equation}\label{term2}
	\begin{aligned}
		\mu^z_j \big[a(U&-\Uhm,s_j-\sjhm)-c_h(\Uhm,\sjhm) \big]
		= \mu^z_j\, a(U-\Uhm,s_j) \\ &\leqslant |\mu^z_j|\,(U-\Uhm,-\Delta s_j)
		\leqslant |\mu^z_j|\,\|U-\Uhm\|_{0,-\ta} \|\Delta s_j\|_{0,\ta}\\
		&\lesssim h^{k+1} \|u-\uhm\|_{0, \beta(\varkappa)} \|U\|_{k+1,-\ta}
		\lesssim h^{k+1} \|f\|_{k-1,-\ta} \|u-\uhm\|_{0, \beta(\varkappa)}.
	\end{aligned}
	\end{equation}
	Note that, in general, $Z_{\varkappa}\not\in H^{k+1}_{-\ta}(\Omega)$, and thus the similar arguments can not be used for the third term. Instead, we use the $\varkappa$-gradient bound \cref{gerr} which are valid also for $s_i-\sihm$, together with the modified Galerkin orthogonality \cref{GO}, approximation properties of $\Iwhk$, and \cref{aed}:	
	\begin{equation}\label{term3}
	\begin{aligned}
		\mu^u_i \big[a(s_i&-\sihm,Z_{\varkappa}-\Zhmk)-c_h(\sihm,\Zhmk)\big] = \mu^u_i a(s_i-\sihm,Z_{\varkappa}-\Iwhk Z_{\varkappa})\\
		&\leqslant |\mu^u_i|\|\g(s_i-\sihm)\|_{0,\beta(\varkappa)}\|\g(Z_{\varkappa}-\Iwhk Z_{\varkappa})\|_{0,-\beta(\varkappa)}\\
		&\lesssim h^{\varkappa+1} \|f\|_{k-1,-\ta}\|Z_{\varkappa}\|_{2,-\beta(\varkappa)} \lesssim h^{\varkappa + 1} \|f\|_{k-1,-\ta}  \| u-\uhm\|_{0, \beta(\varkappa)}.
	\end{aligned}
	\end{equation}
	The last term in \cref{derr} is the most regular one, and thus, using the estimates for the smooth remainder from \cref{l:srem} together with the smoothness of $Z_{\varkappa}$, we obtain:
	\begin{equation}\label{term4}
	\begin{aligned}
		a(U&-\Uhm,Z_{\varkappa}-\Zhmk)-c_h(\Uhm,\Zhmk) \\
		&= a(U -\Uhm, Z_{\varkappa}) = (U-\Uhm,-\Delta Z_{\varkappa}) \leqslant \| U-\Uhm\|_{0,-\ta} \|\Delta Z_{\varkappa} \|_{0,\ta}\\
		&\leqslant h^{k+1} \|U\|_{k+1,-\ta}\|Z_{\varkappa}\|_{2,-\beta(\varkappa)} \lesssim h^{k+1} \|f\|_{k-1,-\ta} \|u-\uhm\|_{0,\beta(\varkappa)}.
	\end{aligned}
	\end{equation}
	Summing all four terms together, i.e., \cref{term1}--\cref{term4}, and dividing by $\|u-\uhm\|_{0,\beta(\varkappa)}$ we obtain the $L^2_{\beta(\varkappa)}-$error for the inductive step:
	\begin{align}\label{errkappa}
		\|u-\uhm\|_{0,\beta(\varkappa)} \lesssim h^{\varkappa + 1} \|f\|_{k-1,-\ta}.
	\end{align}
	
	The proof of \cref{weiconv} is then finished for $\varkappa = \sigma$.
\end{proof}

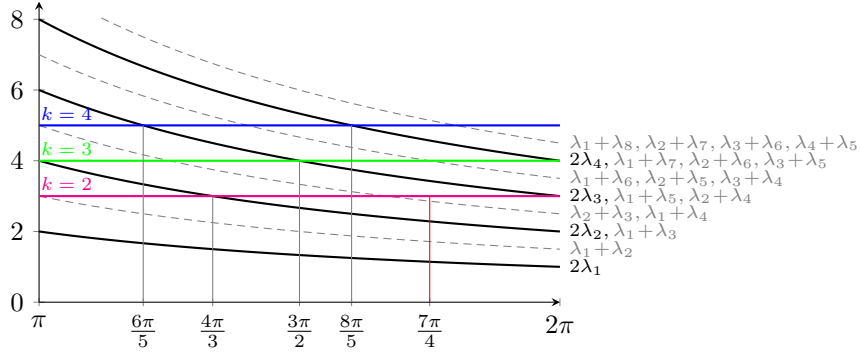
\begin{figure}
	\begin{center}
	\begin{tikzpicture}[scale=1.]
		\def\piv{3.1415926535897932385}
		\def\dpiv{6.283185307179586477}
		\def\ca{3.769911184}
		\def\cb{4.188790205}
		\def\cc{4.71238898}
		\def\cd{5.026548246}
		\begin{axis}[
			yscale=0.7,ymax=8.5,ymin=0,xmin=pi,xmax=2*pi,xtick={\piv,\ca,\cb,\cc,\cd,5.497787144,\dpiv},
			xticklabels={$\pi$,$\frac{6\pi}{5}$,$\frac{4\pi}{3}$,$\frac{3\pi}{2}$,$\frac{8\pi}{5}$,$\frac{7\pi}{4}$,$2\pi$},axis x line=bottom,axis y line=left, ylabel={},
			every axis x label/.style={at={(ticklabel* cs:1.05)},anchor= north east},clip=false
			]
		\addplot[domain=pi:2*pi,color=black,thick,samples=100] {2*pi/x};\addplot[domain=pi:2*pi,color=gray,densely dashed,samples=100] {3*pi/x};
		\addplot[domain=pi:2*pi,color=black,thick,samples=100] {4*pi/x};\addplot[domain=pi:2*pi,color=gray,densely dashed,samples=100] {5*pi/x};
		\addplot[domain=pi:2*pi,color=black,thick,samples=100] {6*pi/x};\addplot[domain=pi:2*pi,color=gray,densely dashed,samples=100] {7*pi/x};
		\addplot[domain=pi:2*pi,color=black,thick,samples=100] {8*pi/x};\addplot[domain=1.12*pi:2*pi,color=gray,densely dashed,samples=100] {9*pi/x};
		\addplot[domain=pi:2*pi,color=magenta,thick] {3};\addplot[domain=pi:2*pi,color=green,thick] {4};\addplot[domain=pi:2*pi,color=blue,thick] {5};
		\draw[thin,color=gray] (axis cs:\ca,0) -- (axis cs:\ca,5);\draw[thin,color=gray] (axis cs:\cb,0) -- (axis cs:\cb,3);
		\draw[thin,color=gray] (axis cs:\cc,0) -- (axis cs:\cc,4);\draw[thin,color=gray] (axis cs:\cd,0) -- (axis cs:\cd,5);
		\draw[thin,color=red] (axis cs:5.497787144,0) -- (axis cs:5.497787144,3);

		\node [anchor=west] at (axis cs:\dpiv,1) {\scriptsize{$2\lambda_1$}};
		\node [anchor=west,gray] at (axis cs:\dpiv,1.5) {\scriptsize{$\lambda_1\!+\!\lambda_2$}};
		\node [anchor=west] at (axis cs:\dpiv,2) {\scriptsize{$2\lambda_2, \color{gray}{\lambda_1\!+\!\lambda_3}$}};
		\node [anchor=west,gray] at (axis cs:\dpiv,2.5) {\scriptsize{$\lambda_2\!+\!\lambda_3, \lambda_1\!+\!\lambda_4 $}};
		\node [anchor=west] at (axis cs:\dpiv,3) {\scriptsize{$2\lambda_3,\color{gray}{\lambda_1\!+\!\lambda_5, \lambda_2\!+\!\lambda_4}$}};
		\node [anchor=west,gray] at (axis cs:\dpiv,3.5) {\scriptsize{$\lambda_1\!+\!\lambda_6, \lambda_2\!+\!\lambda_5,\lambda_3\!+\!\lambda_4$}};
		\node [anchor=west] at (axis cs:\dpiv,4) {\scriptsize{$2\lambda_4,\color{gray}{ \lambda_1\!+\!\lambda_7, \lambda_2\!+\!\lambda_6, \lambda_3\!+\!\lambda_5}$}};
		\node [anchor=west,gray] at (axis cs:\dpiv,4.5) {\scriptsize{$\lambda_1\!+\!\lambda_8, \lambda_2\!+\!\lambda_7,\lambda_3\!+\!\lambda_6,\lambda_4\!+\!\lambda_5$}};
		\node [anchor=south west,magenta] at (axis cs:\piv-0.045,3-0.1) {\scriptsize{$k=2$}};
		\node [anchor=south west,green] at (axis cs:\piv-0.045,4-0.1) {\scriptsize{$k=3$}};
		\node [anchor=south west,blue] at (axis cs:\piv-0.045,5-0.1) {\scriptsize{$k=4$}};
		\end{axis}
	\end{tikzpicture}
	\end{center}
	\caption{Graphical illustration of the total number of singular functions which pollution have to be controlled (bold black) with respect to the approximation order $k$, in virtue of \cref{l:aprconv2}. Dashed lines represent terms of pollution which converge a~priori fast enough if \cref{U}-property is fulfilled, as well as the combination of the singular functions as depicted in gray. The colored lines represent $k+1$ bounds.}
	\label{F:dist2}
\end{figure}

Note, for $\sigma = k$, that the assumptions \cref{assums} in \cref{weiconv} are formulated in a stronger way than those from the main theorem. Namely, we have assumed for all combinations of $i,j$ that
\begin{align*}
	a(s_i-\sihm,s_j-\sjhm) - c_h(\sihm,\sjhm) = \mathcal{O}(h^{k+1}).
\end{align*}
What we show now is that some of the terms from \cref{assums} converge automatically with the prescribed rate under the \cref{U}-assumption. Namely, as we will see below, the terms for which $i\neq j$. The terms with $i=j$ reflect the essential pollution, in our setting the only contribution to the global pollution. 

We shall carry out the proof of the statement above again in a recursive fashion. For readability, we distinguish between several cases, each is separately studied in different lemmas or sections:\\
\begin{tabular}{l l}
	& \\[-0.5em]
	\cref{l:aprconv1}: & $\lambda_i+\lambda_j\geqslant k+1$: both singular functions have sufficient regularity.\\[0.3em]
	\cref{l:aprconv2}, \ref{l:gradhigh}: & $\lambda_i+\lambda_j< k+1$ and $i\neq j$: $s_i$, $s_j$ are mutually orthogonal.\\[0.3em]
	\Cref{S:gamma}: & $\lambda_i+\lambda_j< k+1$ and $i = j$: essential pollution $g_h(s_i)$.\\[0.3em]
\end{tabular}

\begin{lemma}\label{l:aprconv1}
	For $\lambda_i+\lambda_j\geqslant k+1$ it holds:
	\begin{align}\label{aprconvstr}
		a(s_i-\sihm,s_j-\sjhm) - c_h(\sihm,\sjhm)  \lesssim h^{k+1}.
	\end{align}
\end{lemma}
\begin{proof}
	If we rewrite the $c_h(\cdot,\cdot)$ term by
	\begin{align*}
		c_h(\sihm,\sjhm)=c_h(s_i-\sihm,s_j-\sjhm) + c_h(s_i,s_j) - c_h(s_i,s_j-\sjhm) - c_h(s_j,s_i-\sihm),
	\end{align*}
	we can use the Cauchy--Schwarz inequality, direct integration of $s_i$, $s_j$ on $\Sh$ and \cref{l:standest}, and thus we obtain:
	\begin{align*}
		|&a(s_i-\sihm,s_j-\sjhm) - c_h(\sihm,\sjhm)|\\
		&\lesssim \|\g(s_i-\sihm)\|_0\|\g(s_j-\sjhm)\|_0 + \|\g s_i\|_{L^2(\Sh)}\|\g s_j\|_{L^2(\Sh)}\\
		&\quad+\|\g(s_i-\sihm)\|_0\|\g s_j\|_{L^2(\Sh)}+\|\g(s_j-\sjhm)\|_0\|\g s_i\|_{L^2(\Sh)}
		\lesssim h^{\lambda_i+\lambda_j}\lesssim h^{k+1}.
	\end{align*}
	Note also that generally for $\lambda_i+\lambda_j \geqslant \sigma+1$, the convergence order in \cref{aprconvstr} is $\sigma+1$ at least.
\end{proof}

Now, let us take a closer look on the case when \cref{assums} contains a combination of different singular functions, i.e., $i\neq j$.

\begin{lemma}\label{l:aprconv2}
	Let $\lambda_i+\lambda_j< k+1$, the \cref{U}-assumption \cref{assint} holds, and assume also that for some $\sigma\leqslant k$ and $i <(\sigma+1)\frac{\omega}{2\pi}$ it is valid:
	\begin{align}\label{assopt}
		g_h(s_i) = \mathcal{O}(h^{\sigma+1}) \qquad \text{and} \qquad \|s_i-\sihm\|_{0,\Oc{1}} \lesssim h^\sigma.
	\end{align}
 Then
	\begin{align}\label{aprconvstr2}
		a(s_i-\sihm,s_j-\sjhm) - c_h(\sihm,\sjhm)  \lesssim h^{\sigma+1}, \qquad i\neq j.
	\end{align}
\end{lemma}

\begin{proof}
	Before we start with the proof, note that the condition $i <(\sigma+1)\frac{\omega}{2\pi}$ identifies the number of the singular functions  for which
 $g_h(s_i)$ contribute to the global pollution  $g_h(u)$ and thus need to be a~priori corrected by a suitable choice of $c_h(\cdot,\cdot)$. If $\sigma = k$, it is also equivalent to $2\lambda_i <(k+1)$, see \cref{F:dist2}, bold lines. 

	We assume, without loss of generality due to the symmetry of $a_h(\cdot,\cdot)$, that $i\leqslant j$. The character of the problem for the higher order approximation is similar to the linear one. This time, however, instead of local symmetry-properties of $s_i$ and $s_j$, we focus on their mutual orthogonality. Again, under the assumption that the modification is free of the pollution in the very interior of the domain, we shall use a localization strategy, and aim to prove that despite the defect in the orthogonality relationship due to the approximation, the error is still small enough.

	Thus, we define a suitable decomposition of $\Omega$, see \cref{f:rec}, such that
	\begin{align*}
		\Omega_3\subset \Omega_{2}\subset \Omega_{1}\subset \Omega_0=\Omega, \qquad \text{dist}(\partial \Omega_I\setminus\partial\Omega,\partial\Omega_{I-1}\setminus\partial\Omega)>0, \quad I\leqslant 3,
	\end{align*}
	with $|\Omega_{I}| =\mathcal{O}(1)$, $\mathbf{x}_s\in\overline\Omega_I$, $\mathbf{x}_s$ denoting the singular point, and the boundaries of $\Omega_I$ matching the initial mesh. Further, we define a cut-off function
	\begin{align*} 
		\chi(r) = 
		\begin{cases}
		1 &\quad \mbox{on } \Oc{1},\\
		0 &\quad \mbox{on } \Omega_2,\\
		\end{cases}
	\end{align*}
	and its complement to one, $1-\chi$, see \cref{f:cutoffs}, both being functions of $r$, only.
	\begin{figure}
		\centerline{\includegraphics[trim=0 103 0 80,clip,width=0.8\textwidth]{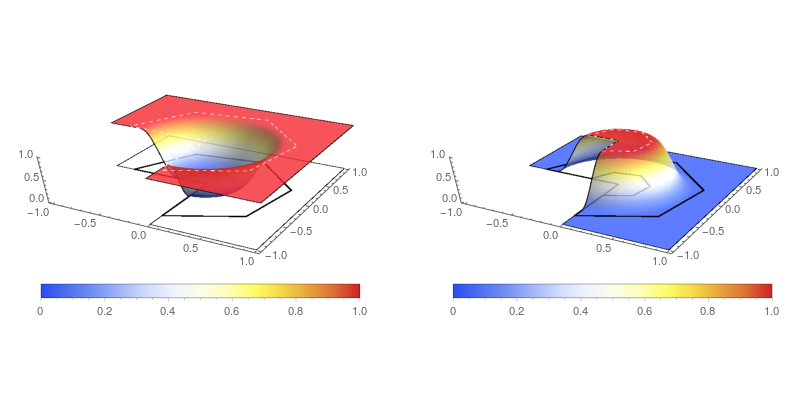}}
		\caption{Left: Cut-off $\chi$, Right: Cut-off $(1-\chi)$. Both represented on the L-shape domain with boundaries of $\Omega_1$ and $\Omega_2$.}
		\label{f:cutoffs}
	\end{figure}
	Due to the fact that $\chi=0$ on $\Sh$, we can by the modified Galerkin orthogonality decompose the term of interest into two parts, namely
	\begin{equation}\label{tott2}
	\begin{aligned}
		a(s_i &- \sihm,s_j - \sjhm) - c_h(\sihm,\sjhm) = a(s_i-\sihm, s_j)\\
		&= a(s_i-\sihm, \chi s_j- \Ihk (\chi s_j)) +  a(s_i-\sihm,  (1-\chi)s_j) = (\text{T1}) + (\text{T2}).
	\end{aligned}
	\end{equation}
	By the Cauchy-Schwarz inequality and the fact that $\chi s_j- \Ihk (\chi s_j) = 0$ on $\Omega_2$, we can estimate the first term (T1):
	\begin{align}\label{todec2}
		a(s_i-\sihm, \chi s_j- \Ihk (\chi s_j)) \lesssim \|s_i-\sihm\|_{1,\Oc{2}} \|\chi s_j- \Ihk (\chi s_j)\|_{1,\Oc{2}},
	\end{align}
	where 
	\begin{align*}
		\|\chi s_j- \Ihk (\chi s_j)\|_{1,\Oc{2}} \lesssim h^k,
	\end{align*}
	since $\chi s_j$ is sufficiently smooth on $\Oc{2}$. The approximation error term in \cref{todec2} can be bounded by the Wahlbin estimate \cref{whlb1}. More precisely,
	\begin{align}\label{rek}
		\|s_i -\sihm\|_{1,\Oc{2}} \lesssim \|s_i - \Ihk s_i\|_{1,\Oc{3}} + \|s_i-\sihm\|_{0}\lesssim h^{\lambda_i+\lambda_1-\varepsilon} \lesssim h,
	\end{align}
	since the singular functions are smooth away from the singularity and \cref{l:standest}. Thus, all together, we obtain for (T1):
	\begin{align}\label{oto2}
		(\text{T1})\lesssim h^{k+1}.
	\end{align}
	Let us turn our attention to (T2). Due to to the mutual orthogonality between $s_i$ and $s_j$, $i\neq j$, and the dependence of $\chi$ only on the radial component, we have that
	\begin{align}\label{exporth}
		a(s_i,(1- \chi) s_j)= 0.
	\end{align}
	Nevertheless, in the general case, the approximations will not preserve the mutual orthogonality (or even symmetry) properties, i.e.,
	\begin{align}\label{nonorth}
		-a(\sihm,(1-\chi)s_j)\neq 0.
	\end{align}
	We proceed now as follows. First, we define a new finite element function, localized on $\Omega_1$, and with its help, we estimate the convergence rate of the left-hand side in \cref{nonorth}. Namely, let us define a finite element space $\tilde{V}_h\coloneqq \{v_h:\, v_h=w_h|_{\Omega_1}, \, w_h\in V_h\}$ and $q_h\in \tilde{V}_h$ solving a localized discrete Poisson problem:
	\begin{equation}\label{defgh2}
		\begin{aligned}
			a_h(q_h,\vh) & = a_h(\sihm,\vh) \qquad& &\forall\vh\in\tilde{V}_h\cap H^1_0(\Omega_1),\\
			q_h &= \Ihk s_i \qquad& &\text{on } \partial\Omega_1.
		\end{aligned}
	\end{equation}
	Due to the boundary conditions on $q_h$ on $\partial\Omega_1$ and the fact that $s_j$ is a harmonic function in $\Omega_1$, we have by assumption \cref{assint}:
	\begin{align}\label{intort}
		\int_{\Omega_1} \nabla s_j \cdot \nabla q_h \, dx = \int_{\partial\Omega_1\setminus\partial\Omega} (\nabla s_j\cdot \mathbf{n})\,  \Ihk s_i \, dx = \int_{\Omega_1} \nabla s_j\cdot \nabla \Ihk s_i\, dx=0,
	\end{align}
	Then, by \cref{exporth}, \cref{intort}, the dependence of $1-\chi$ only on the radial direction, and the definition of the system \cref{defgh2} with $\vh=\Iwhk ((1-\chi)s_j)$,  we can rewrite (T2) as follows:
	\begin{equation}\label{ttdec2}
		\begin{aligned}
		(\text{T2}) &= a(s_i-\sihm,(1-\chi)s_j) =-a(\sihm,(1-\chi)s_j) = a(q_h - \sihm,(1-\chi)s_j)\\
		&= a\big(q_h - \sihm,(1-\chi)s_j - \Iwhk ((1-\chi)s_j)\big).
		\end{aligned}
	\end{equation}
	Since $q_h-\sihm$ is a discrete harmonic on $\Omega_1$, its $H^1$-norm can be  bounded by its trace norm:
	\begin{equation}\label{t1tt2}
	\begin{aligned}
		\|q_h-\sihm\|_{1,\Omega_1} &\lesssim \|q_h-\sihm\|_{H^{\frac{1}{2}}(\partial \Omega_1)}
		\lesssim\|s_i-q_h\|_{H^{\frac{1}{2}}(\partial \Omega_1)} + \|s_i-\sihm\|_{H^{\frac{1}{2}}(\partial \Omega_1)}\\
		&\lesssim \|s_i-\Ihk s_i\|_{1,\Oc{1}} + \| s_i-\sihm\|_{1,\Oc{1}} \lesssim h^\sigma,
	\end{aligned}
	\end{equation}
	where we have used the triangle inequality in the second step, the equivalence of $H^{\frac{1}{2}}(\partial\Omega_1)-$ and $H^{\frac{1}{2}}_{00}(\partial\Omega_1\setminus\partial\Omega)-$norms for functions from $H^1_0(\Omega)$ together with the trace theorem in the third one, and the assumption~\cref{assopt} with interpolation estimates away from the singularity in the last one.
	Also, the interpolation error in $\Omega_1$ is, due to \cref{estinttilde2} and $\lambda_j\geqslant 1$ due to $j\geqslant2$, as follows:
	\begin{align}\label{t2tt2}
		\|(1-\chi)s_j&-\Iwhk((1-\chi)s_j)\|_{1,\Omega_1}\lesssim h^{\lambda_j}\lesssim h.
	\end{align}
	Both auxiliary results \cref{t1tt2} and \cref{t2tt2} inserted into \cref{ttdec2} imply $(\text{T2}) \lesssim h^{\sigma+1}$, which together with \cref{oto2} concludes the proof.
	Note also that if the assumption \cref{assopt} is fulfilled for $\sigma=k$, we obtain the optimal result $(\text{T2}) \lesssim h^{k+1}$.
\end{proof}

What we show in the next lemma is, that the assumption on the errors on $\Oc{1}$ is actually fulfilled if the correction is suitable, namely we have the following.
\begin{lemma}\label{l:gradhigh}
	Under the assumption $ g_h(s_i)=\mathcal{O}(h^{k+1})$, $i\leqslant K$, the approximation error has the local property on $\Oc{1}$:
	\begin{align}\label{gradhigh}
		\|s_i-\sihm\|_{0,\Oc{1}} \lesssim h^{k}.
	\end{align}
\end{lemma}

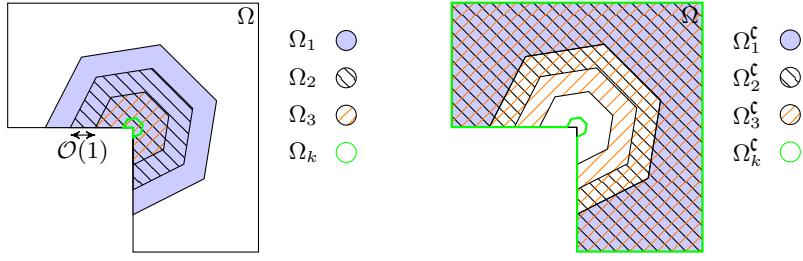
\begin{figure}
	\newcommand{\rec}[1]{plot coordinates{(0,0) (0,-#1) (#1,-#1) (#1,#1) (-#1,#1) (-#1,0) (0,0)}}
	\newcommand{\pac}[1]{plot coordinates{(0., 0.) (0., {-#1}) ({0.809017*#1}, {-0.587785*#1}) ({0.951057*#1}, {0.309017*#1}) ({0.309017*#1}, {0.951057*#1}) ({-0.587785*#1}, {0.809017*#1}) (-#1, 0.) (0., 0.)}}
	\newcommand{\paccom}[2]{plot coordinates{(0., {-#1}) ({0.809017*#1}, {-0.587785*#1}) ({0.951057*#1}, {0.309017*#1}) ({0.309017*#1}, {0.951057*#1}) ({-0.587785*#1}, {0.809017*#1}) (-#1, 0.) (-#2, 0.) (-#2,#2) (#2,#2) (#2,-#2) (0,-#2) (0,-#1)}}
	\newcommand{\distw}{5.25pt}
	\newcommand{\diste}{5.25pt}
	\begin{center}
	\begin{tikzpicture}[scale=0.33,font=\small]
		\draw[color=black]\rec{5};
		\draw[color=black,fill=blue!20]\pac{3.5};
		\draw[color=black,pattern=my north west lines, line space=\distw, pattern color=black]\pac{2.5};
		\draw[color=black,pattern=my north east lines, line space=\diste,pattern color = orange]\pac{1.5};
		\draw[color=green,line width=0.8pt]\pac{0.4};
		\node at (4.5,4.5) {$\Omega$};
		\draw (5.5,3.5) node[right]{$\ \Omega_1$};
		\draw[color=black,fill=blue!20] (8.5,3.5) circle (0.4);
		\draw (5.5,2) node[right]{$\ \Omega_{2}$};
		\draw[color=black,pattern=my north west lines, line space=5pt, pattern color=black] (8.5,2) circle (0.4);
		\draw (5.5,0.5) node[right]{$\ \Omega_{3}$};
		\draw[color=black,pattern=my north east lines, line space=5pt,pattern color = orange] (8.5,0.5) circle (0.4);
		\draw (5.5,-1) node[right]{$\ \Omega_{k}$};
		\draw[solid,green] (8.5,-1) circle (0.4);
		\draw[black, <->] (-2.5,-0.3) -- (-1.5,-0.3);
		\node at (-2,-1) {$\mathcal{O}(1)$};
	\end{tikzpicture}
	\hspace*{1cm}
	\begin{tikzpicture}[scale=0.33,font=\small]
		\draw[preaction={fill=blue!20},color=black,pattern=my north west lines, line space=\distw]\rec{5};
		\draw[color=black,pattern=my north east lines, line space=\diste]\rec{5};
		\draw[preaction={fill=white},color=black,pattern=my north west lines, line space=\distw]\pac{3.5};
		\draw[color=black,pattern=my north east lines, line space=\diste]\pac{3.5};
		\draw[preaction={fill=white},color=black,pattern=my north east lines, line space=\diste]\pac{2.5};
		\draw[color=black,fill=white]\pac{1.5};
		\draw[color=green,line width=0.8pt]\paccom{0.4}{5};
		\node at (4.5,4.5) {$\Omega$};
		\draw (5.5,3.5) node[right]{$\ \Oc{1}$};
		\draw[color=black,fill=blue!20] (8.5,3.5) circle (0.4);
		\draw (5.5,2) node[right]{$\ \Oc{2}$};
		\draw[color=black,pattern=my north west lines, line space=5pt, pattern color=black] (8.5,2) circle (0.4);
		\draw (5.5,0.5) node[right]{$\ \Oc{3}$};
		\draw[color=black,pattern=my north east lines, line space=5pt,pattern color = orange] (8.5,0.5) circle (0.4);
		\draw (5.5,-1) node[right]{$\ \Oc{k}$};
		\draw[solid,green] (8.5,-1) circle (0.4);
	\end{tikzpicture}
	\end{center}
	\caption{Domain decomposition into several embedded sub-domains $\Omega_I$, $I\leqslant k$ (left), and their complements to $\Omega$ (right). Note that $|\Omega_I|=\mathcal{O}(1)$.}
	\label{f:rec}
\end{figure}

\begin{proof}
	First, let us define for each approximation order $k$ a further decomposition of $\Omega$, see \cref{f:rec}, such that
	\begin{align*}
		\Omega_k\subset \Omega_{k-1}\subset...\subset \Omega_1\subset \Omega_0=\Omega, \qquad \text{dist}(\partial \Omega_I\setminus\partial\Omega,\partial\Omega_{I-1}\setminus\partial\Omega)>0, \, I\leqslant k,
	\end{align*}
	with $|\Omega_{I}| =\mathcal{O}(1)$, $\mathbf{x}_s\in\overline\Omega_I$, and the boundaries of $\Omega_I$ matching the initial mesh. 
	
	In the linear approximation case, i.e., $k=1$, there is nothing to prove due to \cref{rek} and $\|s_i -\sihm\|_{1,\Oc{1}} \leqslant \|s_i - \sihm\|_{1,\Oc{2}}$. Let us thus turn our attention to a general approximation order $k$. By \cref{standest}, similarly as in \cref{rek}, we have for $\sihm\in\Vhk$:
 	\begin{align*}
		\|\nabla(s_i -\sihm)\|_{0,\Oc{k-1}} \lesssim \|\nabla(s_i - \Ihk s_i)\|_{0,\Oc{k}} + \|s_i-\sihm\|_{0,\Oc{k}}\lesssim h.
	\end{align*}
	Using this in \cref{assopt} with $\sigma=1$, we can then also show by \cref{l:aprconv2} that at least it holds:
	\begin{align*}
		a(s_i-\sihm,s_j-\sjhm) - c_h(\sihm,\sjhm)  \lesssim h^{\sigma+1}=h^2, \qquad i,j\leqslant N_u, \qquad i\neq j.
	\end{align*}
	In combination with \cref{weiconv},  we now get:
 	\begin{align}\label{oc1}
		\|s_i -\sihm\|_{0,\Oc{k-1}} \lesssim h^2.
	\end{align}
	For $k>2$ we continue by a general inductive step, i.e., we aim to show that \mbox{$\|s_i-\sihm\|_{0,\Oc{k-l+2}}\lesssim h^{l-1}$} implies $\|s_i-\sihm\|_{0,\Oc{k-l+1}}\lesssim h^{l}$.
	This we obtain by the Wahlbin estimate \cref{whlb1}, namely
 	\begin{align*}
		\|\nabla(s_i -\sihm)\|_{0,\Oc{k-l+1}} \lesssim \|\nabla(s_i - \Ihk s_i)\|_{0,\Oc{k-l+2}} + \|s_i-\sihm\|_{0,\Oc{k-l+2}}\lesssim h^{l-1},
	\end{align*}
	and again, we can use \cref{l:aprconv2} and~\ref{weiconv}, this time with $\sigma=l-1$, to show 
 	\begin{align}\label{oc2}
		\|s_i -\sihm\|_{0,\Oc{k-2}} \lesssim h^{\sigma+1}=h^l.
	\end{align}
	The induction steps end as we reach the actual interpolation convergence order, i.e.,
 	\begin{align*}
		\|\nabla(s_i -\sihm)\|_{0,\Oc{1}} \lesssim \|\nabla(s_i - \Ihk s_i)\|_{0,\Oc{2}} + \|s_i-\sihm\|_{0,\Oc{2}}\lesssim h^{k}.
	\end{align*}
	Finally, by \cref{l:aprconv2} and the optimal convergence of the approximation order in $L^2(\Oc{1})$, we obtain the required \cref{gradhigh}.
\end{proof}

Now, putting together \cref{weiconv}, \ref{l:aprconv1}, \ref{l:aprconv2} and \ref{l:gradhigh} we have finished the proof of the \cref{mainthm}.

\section{Construction of the modification $c_h(\cdot,\cdot)$}\label{S:gamma}

The previous sections were dedicated to the proof of \cref{mainthm}, where the main assumption was \cref{assmthm}, namely $g_h(s_i) = \mathcal{O}(h^{k+1})$, $i<\leqslant K$. This means, choosing a correction term $c_h(\cdot,\cdot)$ in such a way that \cref{assmthm} is satisfied guarantees us the optimal order convergence rates of the modified scheme. We remind that $K$ from \cref{Kdef} refers to the total number of essential pollution functions $g_h(s_i)$ contributing to the total pollution. Also, we define the radial element layers $\Sh^i$ with respect to the re-entrant corner placed at $\mathbf{0}$ as:
\begin{align*}
	\Sh^1 \coloneqq \{T\in\mathcal{T}_h:\ \mathbf{0}\in\partial{T}\},
	\qquad \Sh^{i} \coloneqq \{T\in\mathcal{T}_h:\ \partial T\cap\partial \Sh^{i-1}\neq\emptyset\},\ i=2,..., K.
\end{align*}

Following the basic concept from previous studies on energy-correction methods, we can simply adapt the construction of $c_h(\cdot,\cdot)$ to our higher order case and define:
\begin{align}\label{chR}
	c^R_h(u,v) \coloneqq \sum_{i=1}^{K} \int_{\Sh^i} \gamma_i^R\, \nabla u\cdot \nabla v \, dx,
	\qquad \Sh^R=\text{int}\left(\cup_{i=1}^{K}\overline{\Sh^i}\right),
\end{align}
where the constants $\gamma^R_i$ are called correction parameters. We also call the scheme with $c_h^R$ as a method with (radial) layer-correction, since $K$-layers of elements $\Sh^i$ are involved in the definition of the modification.  It is clear that with higher $k$ the support of $c_h^R(\cdot,\cdot)$ is significantly larger. We also present an alternative approach of a function-correction for which, in contrast with the layer-modification, we need only a single layer-support independently on $k$ or $\omega$. Namely, we define:
\begin{align}\label{chF}
	c^F_h(u,v) \coloneqq \sum_{i=1}^{K} \int_{\Sh^1} \gamma^F_i \hat{r}^{i-1} \nabla u\cdot \nabla v \,dx,
	\qquad \Sh^F=\Sh^1,
\end{align}
where $\hat{r}\in[0,1]$ is the $h$-scaled distance from the origin, i.e., $\hat{r}=r/h$.

We shortly denote the correction parameters by a vector valued  $\gam^{\sharp} = (\gamma^{\sharp}_1,\dots,\gamma^{\sharp}_K)^\top$, $\gam^{\sharp}\in\mathbb{R}^K$, ${\sharp}\in\{R,F\}$, and assume that they are such that \cref{A} is satisfied. For example, in the case of the layer-correction, a sufficient condition is that $\gam^R\in\mathcal{B}^{\infty}_\xi(0)=\{\mathbf{y}\in\mathbb{R}^K:\ \|\mathbf{y}\|_{l^\infty}<\xi\}$ with a fixed $\xi<1$. We will see later in \cref{S:numerics} that all our choices of $\gam^R$ satisfy \cref{A}. We will also show that both types of modification are good choices, if optimal parameters $\gam^R$, $\gam^F$ are known. The basic question which then remains is: How to obtain the optimal correction parameters? Let us describe this topic briefly for the layer-correction $c^R_h(\cdot,\cdot)$ and, for the better readability, skip the index $R$. For the function-correction, we use the same strategy.

We follow the same technique as proposed in \cite{Egger2014} and proven in \cite{Ruede2014} for the linear elements for the Poisson problem and in \cite{John2015} for linear elements (but also vector-valued $\gh(\cdot)$) for the case of Stokes problem. This means, we construct a sequence $\{\gam_h\}_{h>0}\subset\mathcal{B}^{\infty}_\xi(0)$ satisfying certain conditions, such that it converges to a unique asymptotic value $\gam^\star\in\mathcal{B}^{\infty}_{\xi}(0)$, being the optimal correction parameter. Namely, for each admissible $\gam=(\gamma_1,\dots,\gamma_K)^\top$, we define the vector valued pollution function $\gh:\mathcal{B}^{\infty}_{\xi}(0)\rightarrow \mathbb{R}^K$ by its components:
\begin{align}\label{ghg}
	g_{h,i}(\gam) = a(s_i-\Rg s_i,s_i-\Rg s_i) - c_h(\Rg s_i,\Rg s_i), \quad  i=1,...,K,
\end{align}
where $\Rg s_i$ represents the modified-approximative solution to $s_i$, this time explicitly dependent on the choice of $\gam$. If we construct $\{\gam_h\}_{h>0}$ such that
\begin{align}\label{ghg2}
	\gh(\gam_h) = \mathbf{0}, \quad \text{for each } h>0,
\end{align}
we can possibly generalize the proof of \cite{Ruede2014} and \cite{John2015} and show the convergence $\gam_h$ to $\gam^\star$ for $h\rightarrow 0$ with a certain rate. Under the assumption on a suitably fast convergence of $\gam_h$, we can show that for the asymptotic vector $\gam^\star$, the necessary condition \cref{assmthm} on the modification $c_h(\cdot,\cdot)$ on each level is satisfied. The validity of \cref{assmthm} will be stated in the next lemma. For that, let us fist define the following auxiliary sequences:
\begin{alignat*}{2}
	\gam_h^K &\coloneqq (\gamma_{h,1}^K,...,\gamma_{h,K}^K)^\top & \text{ such that } \gh(\gam^K_h) &= \boldsymbol{0}, \\
	\gam_h^K &\rightarrow \gam^K \coloneqq (\gamma_1^K,...,\gamma_K^K)^\top, &&\\
	&\vdots &&\\
	\gam_h^l &\coloneqq (\gamma_{h,1}^l,\dots,\gamma_{h,l}^l,\gamma_{l+1}^K,\dots,\gamma_K^K)^\top \qquad&\text{ such that } g_{h,i}(\gam^l_h) &= 0  \quad \text{ for all } i\leqslant l, \\
	\gam_h^l &\rightarrow \gam^l \coloneqq (\gamma_1^l,\dots,\gamma_{l}^l,\gamma_{l+1}^K,\dots,\gamma_K^K)^\top, && \\
	&\vdots &&\\
	\gam_h^1 &\coloneqq (\gamma_{h,1}^1,\gamma_2^K,...,\gamma_K^K)^\top \qquad&\text{ such that } g_{h,1}(\gam^1_h) &= 0, \\
	\gam_h^1 &\rightarrow \gam^1 \coloneqq (\gamma_1^1,\gamma_2^K,...,\gamma_K^K)^\top. &&
\end{alignat*}

\begin{lemma}\label{l:gamconv}
	Let $\gam_{h}^l$ converge to $\gam^l$ at least with the following rate:
	\begin{align}\label{gamconv}
		\|\gam^l-\gam_{h}^l\|_{l^2} \lesssim h^{k+1-2\lambda_l},
	\end{align}
	for all $l\leqslant K$. Then the condition \cref{assmthm} is fulfilled for correction with $\gam=\gam^K=\gam^\star$.
\end{lemma}
Before we step to the proof, let us mention here general properties of $\gh(\gam)$. Namely, by the modified Galerkin orthogonality and the $h-$uniform a~priori bounds on $\Rg s_i$, we can show the Lipschitz continuity of $g_{h,i}$ on $\mathcal{B}^\infty_{\xi}(0)$:
\begin{align*}
	g_{h,i}(\gam)\!-\!g_{h,i}(\tgam) &= a(s_i,\Rtg s_i)\! -\! a(s_i, \Rg s_i)\\
	&= -\!\sum_{j=1}^K\! (\gamma_j-\tilde\gamma_j)\int_{\Sh^j}\!\!\g\Rg s_i\cdot\g\Rtg s_i\, dx,
\end{align*}
and thus
\begin{align*}
	|g_{h,i}(\gam)-g_{h,i}(\tgam)|\lesssim \|\gam-\tilde\gam\|_{l^1} \|s_i\|_1^2.
\end{align*}
The function $g_{h,i}$ is also differentiable:
\begin{align}\label{grad}
	\frac{\partial g_{h,i}}{\partial \gamma_j} = \lim_{\tilde\gamma_j\rightarrow\gamma_j}\frac{g_{h,i}(\gamma_1,...,\tilde\gamma_j,...,\gamma_K)-g_{h,i}(\gam)}{\tilde\gamma_j-\gamma_j} = -\int_{\Sh^j}\!\!\g\Rg s_i\cdot\g\Rg s_i\, dx.
\end{align}
By generalizing the proof for the Stokes problem form \cite{John2015}, we can even show a bound of the gradient:
\begin{align}\label{boundchj}
	m (\kappa_j h)^{2\lambda_i}\leqslant
	-\frac{\partial g_{h,i}(\gam)}{\partial \gamma_j}
	\leqslant M (\kappa_j h)^{2\lambda_i},
\end{align}
with $\kappa_j$ given such that $\mathcal{B}^2_{(\kappa_{j-1}+\kappa_j)h}\subset \Sh^j\subset\mathcal{B}^2_{(\kappa_{j-1}+\kappa_j+1)h}$, $j=1,...,K$. The mesh around the re-entrant corner is then assumed to have additional properties, reflecting this ball inclusions. For a detailed discussion, see \cite{John2015}. Let us just mention here, that all the meshes we use in our computations satisfy this property, and $\kappa_j=1$ for all $j\leqslant K$.

{\it Proof of \cref{l:gamconv}}.
	We will carry our proof by induction in the index $l$, i.e., we show, that only the first $l$ components of $\gh(\gam^l)$ are of the sufficient rate under the assumption that $ g_{h,i}(\gam^{l-1})=\mathcal{O}(h^{k+1})$ for all $i\leqslant l-1$ and \cref{gamconv}. Also, we will show, that actually $\gam^l=\gam^{l-1}$.

	\noindent$\bullet \ l=1$: We have directly by the continuity of $g_{h,1}$, \cref{gamconv}, Young inequality and \cref{boundchj} with \cref{grad}:
	\begin{equation}\label{stempi}\begin{aligned}
		|g_{h,1}(\gam^1)| &= |g_{h,1}(\gam^1)-g_{h,1}(\gam_h^1)|\\
		&\leqslant |\gamma_{1}^1-\gamma_{h,1}^1| \left(\|\nabla \Rh(\gam^1) s_1\|^2_{L^2(\Sh)}+\|\nabla \Rh(\gam^1_h) s_1\|^2_{L^2(\Sh)}\right) \lesssim h^{k+1}.\end{aligned}
	\end{equation}
	$\bullet \ 1<l\leqslant K$: We assume, that $|g_{h,i}(\gam^{l-1})| \lesssim h^{k+1}$, for $i\leqslant l-1$. Due to differentiability of $\gh$ on $\mathcal{B}^\infty_{\xi}(0)$, we have by the mean value theorem in several variables for some $\boldsymbol{\zeta}\in\mathcal{B}^\infty_{\xi}(0)$:
	\begin{align*}
		g_{h,i}(\gam^l) &= g_{h,i}(\gam^{l-1})+(\gam^{l}-\gam^{l-1})\cdot\nabla_{\!\gam}g_{h,i}(\boldsymbol{\zeta})
	\end{align*}
	Let us assume that $\gam^{l}\neq\gam^{l-1}$. Since $\left|\frac{\partial g_{h,i}(\boldsymbol{\zeta})}{\partial \gamma_j}\right| \geqslant c(j) h^{2\lambda_i}$ by \cref{boundchj}, and from the induction assumption it holds that $g_{h,i}(\gam^{l-1})=\mathcal{O}(h^{k+1})$, we obtain for each $i-$th component of $\gh$, $i\leqslant l-1$, a lower bound:
	\begin{align}\label{contr1}
		|g_{h,i}(\gam^l)|\gtrsim h^{2\lambda_i},
	\end{align}
	where we also used that $2\lambda_i<k+1$. At the same time, we have by the same arguments as in the first induction step in \cref{stempi} and assumption \cref{gamconv}:
	\begin{align*}
		|g_{h,i}(\gam^l)| \lesssim \|\gam^{l}-\gam_h^{l}\|_{l^2}\, h^{2\lambda_i} \lesssim h^{2\lambda_i +(k+1-2\lambda_l)},
	\end{align*}
	which is a contradiction to \cref{contr1} again by the fact that $2\lambda_l<k+1$. This means that $\gam^{l}=\gam^{l-1}$ and thus $g_{h,i}(\gam^l) \lesssim h^{k+1}$ for all $i\leqslant l-1$. 

	From the last step then obviously holds that $\gam_h$ solving \cref{ghg2} converges component-wise to $\gam^K \eqqcolon \gam^\star$ for which $\gh(\gam^\star)=\mathcal{O}(h^{k+1})$, i.e., $g_h(s_i)=\mathcal{O}(h^{k+1})$ for all $i\leqslant K$.
\qed

We note that the sequences $\gam_h^{l}$ and $\gam^l$, respectively, are not constructive in a sense, that the explicit knowledge of the limiting vector $\gam^\star$ is required.  Nevertheless, we will show numerically in \cref{S:numerics}, that the assumption \cref{gamconv} on the convergence rates of $\gam_h^l$ are satisfied, and thus we show the validity of \cref{assmthm} for $\gam^\star$ being constructed as the root of \cref{ghg2}.

\section{Numerical results}\label{S:numerics}

In this section, we show some numerical results for problem \cref{prob} using higher order finite elements with and without energy correction. We illustrate our theoretical results for schemes up to order four, and we always assume an exact solution given by: 
\begin{align*}
	u &:= s_1 + s_2+s_3+s_4 = \sum\nolimits_{i=1}^{4}  r^{\lambda_i}\sin(\lambda_i\theta).
\end{align*}
We note that each term contributes to the total energy with a similar magnitude. Also, according to \cite{HoMeWo15}, an adding of a further term  with a higher regularity to the exact solution $u$ does not influence the convergence rates.

The convergence rates and the $\gam$-convergence studies for the second, third and fourth order scheme are performed on the geometries as depicted in \cref{fig:geoms1}. 

\begin{figure}[h!]
	\begin{center}
	\includegraphics[width=0.17\textwidth]{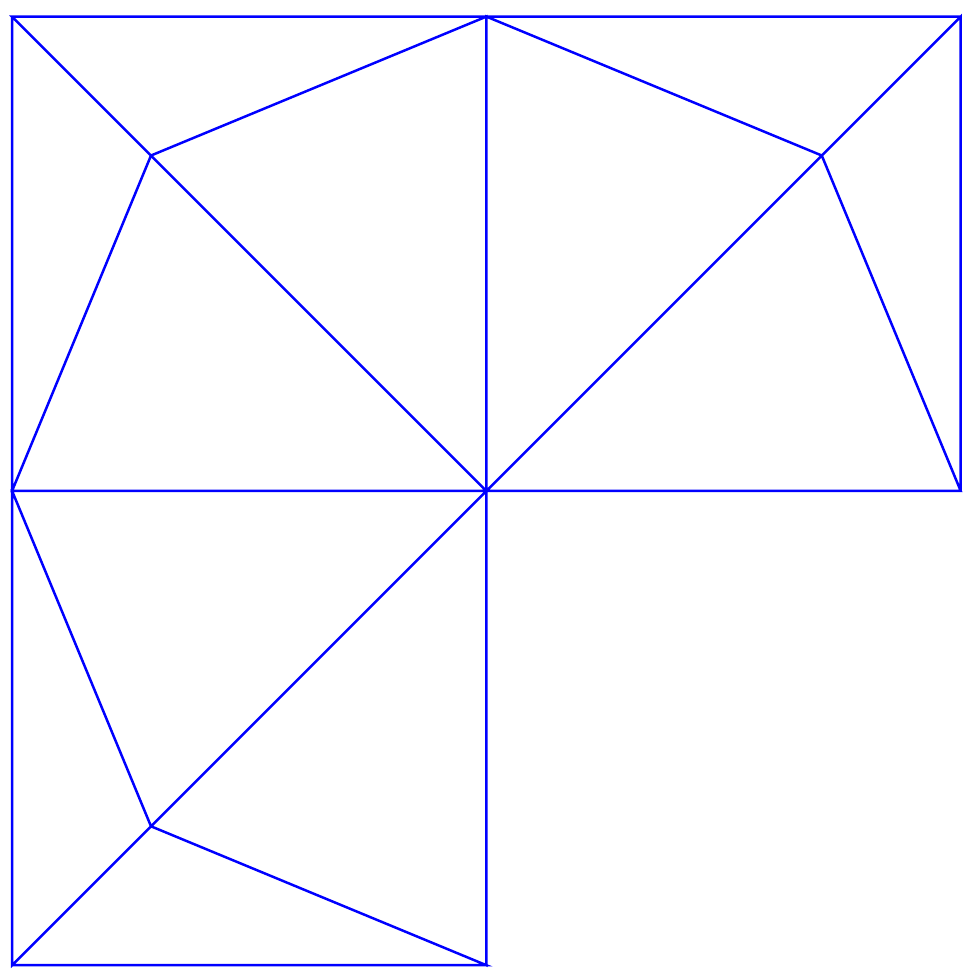}\hspace*{0.8cm}
	\includegraphics[width=0.17\textwidth]{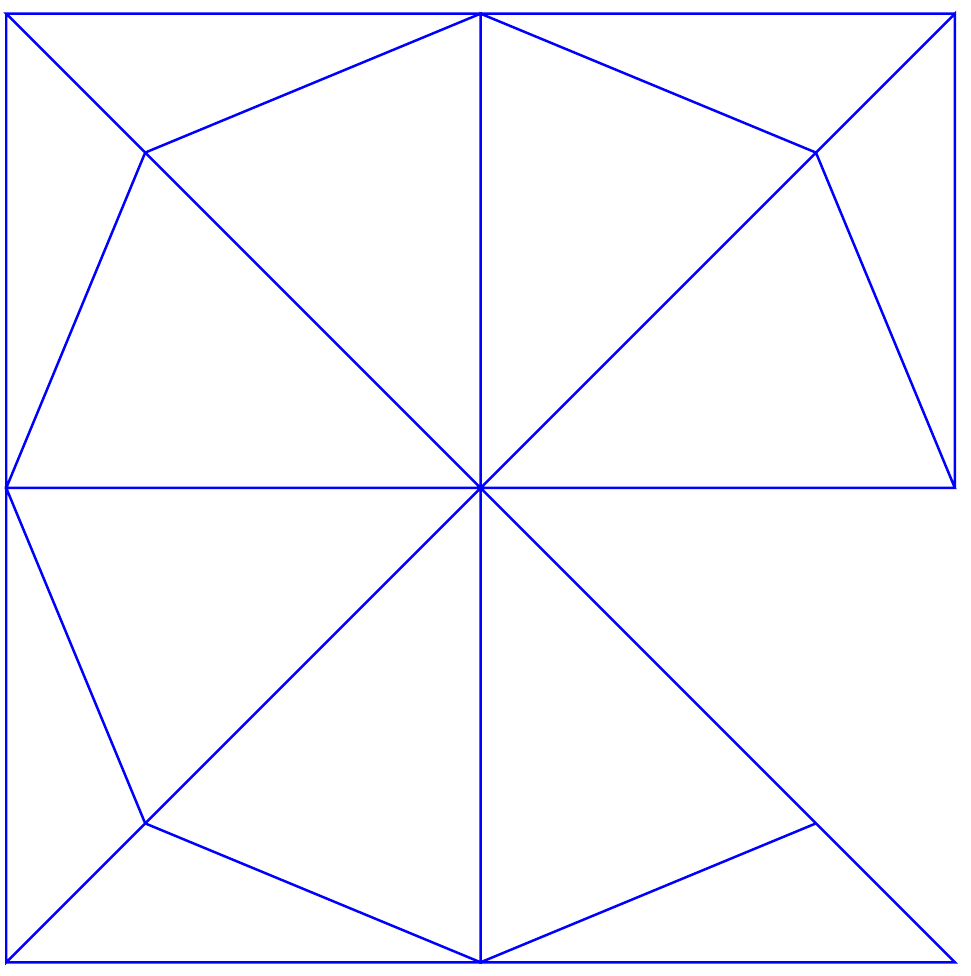}\hspace*{0.8cm}
	\includegraphics[width=0.17\textwidth]{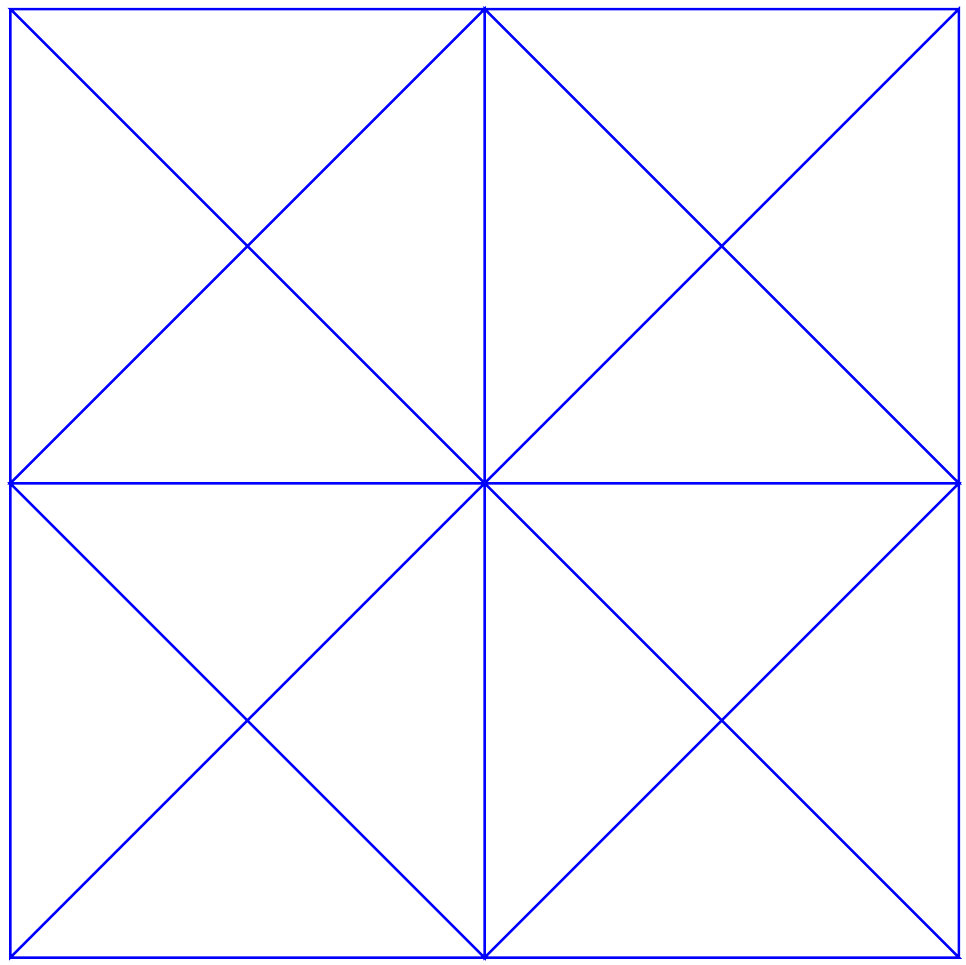}
	\caption{Initial meshes, $\ell=1$ for $\omega = 3/2\pi,\ 7/4\pi, \ 2\pi$ used for $k=2,3,4$ approximations.}
	\label{fig:geoms1}
	\end{center}
\end{figure}

This section is structured as follows: First, we present the results for the energy-correction with the layer-modification \cref{chR}. In \cref{ss:second}, a more detailed study is performed for the second order case, showing convergence results for the modified scheme. Then, in \cref{ss:thirdfourth}, we turn our attention to the results for the third and fourth order approximations, again for the layer correction. In the last part, \cref{ss:chF}, we compare the qualitative behavior of the layer- and function-correction, namely, the influence of the choice of the modification $c_h(\cdot,\cdot)$ on the quantitative error.

All the sections in this part are written exclusively either for the layer- ($c^R_h(\cdot,\cdot)$) or for the function-correction ($c^F_h(\cdot,\cdot)$). For a better readability, we skip the index $R$ and $F$ in the notation of the asymptotic values $\gam^\star$ and their components $\gamma_i^\star$, $i\leqslant K$. Also, the specific values of $\gam^\star$ are stated in the following only up to a certain precision.

Before we start with the numerical accuracy study of the energy corrected approach, we show that our local modification does not deteriorate the condition number of the stiffness-matrix. Namely, in \cref{tab:EW_ratios} we give the ratios $\frac{\hmod{(\lambda_{\min})}}{(\lambda_{\min})_h}$ and $\frac{\hmod{(\lambda_{\max})}}{(\lambda_{\max})_h}$ between eigenvalues of the stiffness matrix which belong to the modified scheme $\hmod{(\cdot)}$ and the standard scheme $(\cdot)_h$, for both correction methods $c_h^R(\cdot,\cdot)$ and $c_h^F(\cdot,\cdot)$.

\begin{table}[!ht]
	\begin{center}
	\begin{tabular}{|c||c|c|c||c|c|c|}\hline
	&\multicolumn{3}{c||}{}&\multicolumn{3}{|c|}{} \\[-0.9em] 
	&\multicolumn{3}{c||}{layer correction $c_h^R(\cdot,\cdot)$}&\multicolumn{3}{|c|}{function correction $c_h^F(\cdot,\cdot)$} \\\hline \hline
	$\ell$ & 2 &3 &4 & 2 &3 &4  \\ \hline
	& & & & & &\\[-0.8em] 
	$\frac{\hmod{(\lambda_{\min})}}{(\lambda_{\min})_h}$ & 	0.9886 &   0.9891 &   0.9952 &  0.9806  &  0.9901  &  0.9954 \\
	& & & & & &\\[-0.8em] 
	$\frac{\hmod{(\lambda_{\max})}}{(\lambda_{\max})_h}$ & 	1.0007 &   1.0000 &   1.0000 &	1.0000 &   1.0000  &  1.0000 \\ \hline
	\end{tabular}
	\label{tab:EW_ratios}
	\caption{Comparison of the condition numbers of the stiffness matrices belonging to corrected and standard scheme; $\omega=\frac74 \pi$, $k=2$.}
	\end{center}
\end{table}

\subsection{Layer-modification $c_h^R(\cdot,\cdot)$: $k=2$}\label{ss:second}
In this subsection, we study the second order scheme for the case of the layer-correction \cref{chR}. We start by the estimation of the asymptotic values of the correction parameters $\gam^\star$, then present the optimal rates obtained with the energy-correction method, and finally we highlight the importance of our symmetry assumptions on the mesh.  

\definecolor{darkgreen}{rgb}{0.125,0.5,0.169}
\def\scal{0.56}
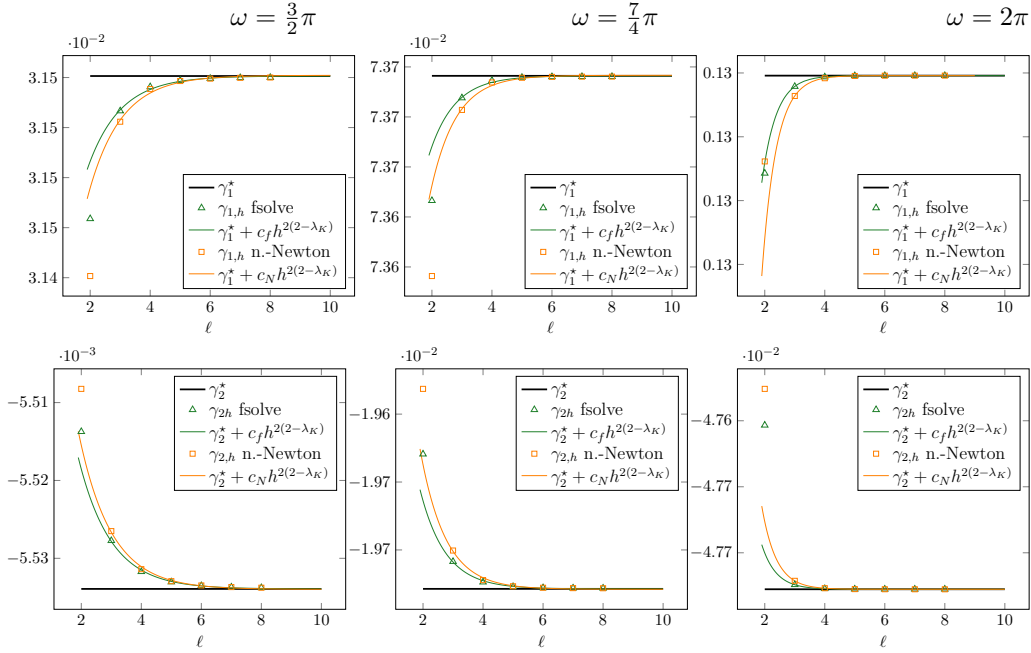
\begin{figure}[h!]
\hspace*{-1.3cm}
\begin{tabular}{rrr}
$\omega = \frac 3 2 \pi$ & $\omega = \frac 7 4 \pi$ & $\omega = 2 \pi$\\
\begin{tikzpicture}[scale = \scal, font = \large]
	\begin{axis}[
		xlabel={$\ell$},legend style={font=\Large},legend pos=south east,legend cell align=left,
		xtick={2,4,6,8,10}
	]	
	\addplot[domain=2:10,color=black,samples=10, line width=1.2pt] {0.031520615381016};
	\addplot[mark=triangle, mark size=2.5pt,color = darkgreen,only marks] coordinates {
	(2,0.0314635639995       )
	(3, 0.0315066508516      ) 
	(4,  0.031516287833      ) 
	(5,   0.0315188324445    )
	(6,    0.0315196378154   ) 
	(7,     0.0315199256197  ) 
	(8,      0.0315200301303 )};	
	\addplot[domain=1.9:10,color=darkgreen,samples=200] {0.031520615381016 -8.584469496399729e-05*2^(-(x-1)*(2*(2 - 1.333333333333333))};
	\addplot[mark=square, mark size=1.7pt,color = orange, only marks] coordinates {
	(2,	0.0314406753164   	 )
	(3, 0.0315023548935   	) 
	(4, 0.0315154129879   	) 
	(5, 0.0315186188331   	)
	(6,  0.0315195740881  	) 
	(7,   0.0315199041496 	) 
	(8,    0.0315200278811	)};	
	\addplot[domain=1.9:10,color=orange,samples=200] {0.031520896788524 -1.133745027882618e-04*2^(-(x-1)*(2*(2 - 1.333333333333333))};	
	\legend{{$\gamma^\star_{1}$}, $\gamma_{1,h}$ fsolve, $\gamma^\star_{1} + c_f h^{2(2-\lambda_K)}$, $\gamma_{1,h}$ n.-Newton, $\gamma^\star_{1} + c_N h^{2(2-\lambda_K)}$}
	\end{axis}
\end{tikzpicture}\hspace*{-0.5cm}
& 
\begin{tikzpicture}[scale = \scal, font = \large]
	\begin{axis}[
		xlabel={$\ell$},
		legend style={font=\Large},legend pos=south east,legend cell align=left,
		xtick={2,4,6,8,10}
	]	
	\addplot[domain=2:10,color=black,samples=10, line width=1.2pt] {0.073696439582743};
	\addplot[mark=triangle, mark size=2.5pt,color = darkgreen,only marks] coordinates {
	(2,		0.07364652219  	 )
	(3, 	0.0736875962101	) 
	(4, 	0.0736945463527	) 
	(5, 	0.0736958000826	)
	(6, 	0.0736960589783	) 
	(7, 	0.0736961208289	) 
	(8, 	0.0736961375651	)};	
	\addplot[domain=1.9:10,color=darkgreen,samples=200] {0.073696439582743 -9.273036178016895e-05*2^(-(x-1)*(2*(2 - 1.142857142857143))};	
	\addplot[mark=square, mark size=1.7pt,color = orange, only marks] coordinates {
	(2,	0.0736163727836		 )
	(3,	0.0736828100025		) 
	(4,	0.0736938012501		) 
	(5,	0.0736956748593		)
	(6,	0.0736960349641		) 
	(7,	0.0736961154123		) 
	(8,	0.073696136512 		)};	
	\addplot[domain=1.9:10,color=orange,samples=200] {0.073696649201095 -1.447593013745391e-04*2^(-(x-1)*(2*(2 - 1.142857142857143))};	
	\legend{{$\gamma^\star_{1}$}, $\gamma_{1,h}$ fsolve, $\gamma^\star_{1} + c_f h^{2(2-\lambda_K)}$, $\gamma_{1,h}$ n.-Newton, $\gamma^\star_{1} + c_N h^{2(2-\lambda_K)}$}
	\end{axis}
\end{tikzpicture}\hspace*{-0.5cm}
&
\begin{tikzpicture}[scale = \scal, font = \large]
	\begin{axis}[
		xlabel={$\ell$}, legend style={font=\Large},legend pos=south east,legend cell align=left,
		xtick={2,4,6,8,10}
	]	
	\addplot[domain=2:10,color=black,samples=10, line width=1.2pt] {0.130969600401589};	
	\addplot[mark=triangle, mark size=2.5pt,color = darkgreen,only marks] coordinates {
	(2,	0.130954303116		 )
	(3, 0.130967895495		) 
	(4, 0.13096938478 		) 
	(5, 0.130969574523		)
	(6, 0.130969598322		) 
	(7, 0.130969601258		) 
	(8, 0.130969599238		)};	
	\addplot[domain=1.9:10,color=darkgreen,samples=200] {0.130969600401589 -1.091328543371559e-04*2^(-(x-1)*(2*(2 - 0.50))};
	\addplot[mark=square, mark size=1.7pt,color = orange, only marks] coordinates {
	(2,	0.130956159539		 )
	(3,	0.130966414131		) 
	(4,	0.130969195048		) 
	(5,	0.130969550725		)
	(6,	0.130969595336		) 
	(7,	0.13096960172 		) 
	(8,	0.130969606332		)};	
	\addplot[domain=1.9:9,color=orange,samples=200] {0.130969601103628 -2.040272452387742e-04*2^(-(x-1)*(2*(2 - 0.50))};
	\legend{{$\gamma^\star_{1}$}, $\gamma_{1,h}$ fsolve, $\gamma^\star_{1} + c_f h^{2(2-\lambda_K)}$, $\gamma_{1,h}$ n.-Newton, $\gamma^\star_{1} + c_N h^{2(2-\lambda_K)}$}
	\end{axis}
\end{tikzpicture}
\\
\begin{tikzpicture}[scale = \scal, font = \large]
	\begin{axis}[
		xlabel={$\ell$},legend style={font=\Large},legend pos=north east,legend cell align=left,
		xtick={2,4,6,8,10}
	]	 
	\addplot[domain=2:10,color=black,samples=10, line width=1.2pt] {-0.005533930763362};  
	\addplot[mark=triangle, mark size=2.5pt,color = darkgreen,only marks] coordinates {
	(2,	-0.00551373272898	 )
	(3, -0.00552773593053	) 
	(4, -0.00553171927646	) 
	(5, -0.00553303934889	)
	(6, -0.00553352256723	) 
	(7, -0.00553370827677	) 
	(8, -0.00553377784001	)};	
	\addplot[domain=1.9:10,color=darkgreen,samples=200] {-0.005533930763362 +3.876054338058586e-05*2^(-(x-1)*(2*(2 - 1.333333333333333))};
	\addplot[mark=square, mark size=1.7pt,color = orange, only marks] coordinates {
	(2,		-0.00550823816954   	 )
	(3,		-0.0055265074879    	) 
	(4,		-0.00553139508916   	) 
	(5,		-0.00553293679298   	)
	(6,		-0.00553348592984   	) 
	(7,		-0.00553369472495   	) 
	(8,		-0.0055337764918		)};		
	\addplot[domain=1.9:10,color=orange,samples=200] {-0.005533986868577 +4.661679370359960e-05*2^(-(x-1)*(2*(2 - 1.333333333333333))};
	\legend{{$\gamma^\star_{2}$}, $\gamma_{2h}$ fsolve, $\gamma^\star_{2} + c_f h^{2(2-\lambda_K)}$, $\gamma_{2,h}$ n.-Newton, $\gamma^\star_{2} + c_N h^{2(2-\lambda_K)}$}
	\end{axis}
\end{tikzpicture} \hspace*{-0.5cm}
& 
\begin{tikzpicture}[scale = \scal, font = \large]
	\begin{axis}[
		xlabel={$\ell$},legend style={font=\Large},legend pos=north east,legend cell align=left,
		xtick={2,4,6,8,10}
	]		
	\addplot[domain=2:10,color=black,samples=10, line width=1.2pt] {-0.019675729648267};
	\addplot[mark=triangle, mark size=2.5pt,color = darkgreen,only marks] coordinates {
	(2,		-0.0196558894622 	 )
	(3, 	-0.0196717113879   	) 
	(4, 	-0.0196747335785 	) 
	(5, 	-0.0196754080983 	)
	(6, 	-0.0196755787023 	) 
	(7, 	-0.0196756259051 	) 
	(8, 	 -0.019675639849 	)};	
	\addplot[domain=1.9:10,color=darkgreen,samples=200] {-0.019675729648267 +4.255653639421613e-05*2^(-(x-1)*(2*(2 - 1.142857142857143))};
	\addplot[mark=square, mark size=1.7pt,color = orange, only marks] coordinates {
	(2,	-0.0196462800991		 )
	(3,	-0.0196700853315		) 
	(4,	-0.0196744403256		) 
	(5,	-0.0196753475592		)
	(6,	-0.0196755642454		) 
	(7,	-0.019675621995 		) 
	(8,	-0.0196756390248		)};		
	\addplot[domain=1.9:10,color=orange,samples=200] {-0.019675789092759 +6.020516458071600e-05*2^(-(x-1)*(2*(2 - 1.142857142857143))};
	\legend{{$\gamma^\star_{2}$}, $\gamma_{2h}$ fsolve, $\gamma^\star_{2} + c_f h^{2(2-\lambda_K)}$, $\gamma_{2,h}$ n.-Newton, $\gamma^\star_{2} + c_N h^{2(2-\lambda_K)}$}
	\end{axis}
\end{tikzpicture} \hspace*{-0.5cm}
&
\begin{tikzpicture}[scale = \scal, font = \large]
	\begin{axis}[
		xlabel={$\ell$},legend style={font=\Large},legend pos=north east,legend cell align=left,
		xtick={2,4,6,8,10}
	]	    
	\addplot[domain=2:10,color=black,samples=10, line width=1.2pt] {-0.047665528247295};
	\addplot[mark=triangle, mark size=2.5pt,color = darkgreen,only marks] coordinates {
	(2,	-0.0476406805185 		 )
	(3, -0.0476648423477 		) 
	(4, -0.0476654428402 		) 
	(5, -0.0476655181513 		)
	(6,  -0.0476655275966		) 
	(7,  -0.0476655287291		) 
	(8,  -0.0476655259956		)};	
	\addplot[domain=1.9:10,color=darkgreen,samples=200] {-0.047665528247295 +4.389427059564194e-05*2^(-(x-1)*(2*(2 - 0.50))};
	\addplot[mark=square, mark size=1.7pt,color = orange, only marks] coordinates {
	(2,		-0.0476351725769	 )
	(3,		-0.0476642512322	) 
	(4,		-0.0476653675416	) 
	(5,		-0.0476655087072	)
	(6,		-0.0476655264055	) 
	(7,		-0.0476655295828	) 
	(8,		-0.0476655345907	)};		 
	\addplot[domain=1.9:10,color=orange,samples=200] {-0.047665529977951 +8.186154374798488e-05*2^(-(x-1)*(2*(2 - 0.50))};
	\legend{{$\gamma^\star_{2}$}, $\gamma_{2h}$ fsolve, $\gamma^\star_{2} + c_f h^{2(2-\lambda_K)}$, $\gamma_{2,h}$ n.-Newton, $\gamma^\star_{2} + c_N h^{2(2-\lambda_K)}$}
	\end{axis}
\end{tikzpicture}
\end{tabular}
\vspace*{-0.3cm}
\caption{Best fits of $\gamma_{i,h}$ obtained by {\tt fsolve} and the nested-Newton scheme, with respect to the theoretical expected rates. First row: fits of $\gamma_1^\star$; Second row: fits of $\gamma_2^\star$; different re-entrant corners are shown from the left to the right}
\label{fig:gamme_calc_O2}
\end{figure}

\subsubsection{$\gamma^\star$ fitting and the Newton method}
An important part of the energy- correction method is an accurate approximation of the correction-parameter $\gam^\star=(\gamma_1^\star,\dots,\gamma_K^\star)$, as a limit vector of the series $\{\gam_h\}_{h>0}$. In our study, we use two different possibilities to estimate $\gam_h$. The first one is purely based on existing FEniCS routines \cite{AlnaesBlechta2015a} and uses {\tt fsolve} to compute an approximation of $\gam_h$ as the root of an implicitly given function. The second one uses the nested-Newton method of \cite{Ruede2014} adapted to the higher order case. The implementation of this method in FEniCS is in comparison to the {\tt fsolve} routine much cheaper in terms of computation time.

With the obtained $\gam_h$, a fit is performed in MATLAB \cite{Matlab2015} for the approximation of $\gam^\star$ as the limit value of $\gam_h, h\rightarrow 0$, by exploiting the asymptotic
\begin{equation}
	\gam^\star + c \, h^{2(k-\lambda_K)},
\end{equation} 
where $c$ is a constant, differing for the fits to data obtained by {\tt fsolve} ($c_f$) or by the nested Newton ($c_N$). At this point it should be noted that the calculated $\gam^\star$ slightly differs depending on the underlying fit data.

\cref{fig:gamme_calc_O2} depicts the resulting fits of $\gamma_i^\star$ for both {\tt fsolve} and nested Newton method. We also show the results for all three basic scenarios of the re-entrant corners, $\omega =3/2\pi$, $7/4\pi$ and $2\pi$. In the upper row of the figure, we present fits for the first component of $\gam^\star$, in the lower row, the fits for the second one.
For both methods, we obtain the same convergence rate for the correction parameters. Both methods tend, as it is expected from our theory, to the same $\gam^\star$. In our case the above mentioned difference between the approximated $\gam^\star$ is less than $10^{-6}$. For our further calculations, we use the $\gam^\star$ obtained by the {\tt fsolve} routine. To obtain the same convergence results, it would also be possible to use $\gam^\star$ obtained with by the nested-Newton method, or alternatively also the first six digits of if it. 

\subsubsection{Convergence rates}
With the asymptotic values $\gam^\star$, we illustrate in \cref{tab: L_Shape_ring_rates} and \ref{tab:angle_shape_ring_rates} the convergence rates with and without energy-correction. Our numerical results are in excellent agreement with the theory, see \cref{mainthm}. The errors are measured for both, the $L^2-$ and the $L^2_{\alpha}-$ norms. In this part, we focus on the case of $\omega = 3/2\pi$ and $\omega = 7/4\pi$. In comparison with the linear approximations, as studied in \cite{Egger2014}, the error in weighted gradient norms for the second order has also a reduced convergence order. Therefore, we present also the convergence rates for the gradient, see \cref{tab: L_Shape_ring_rates} and~\ref{tab:angle_shape_ring_rates}.

\hspace*{-2cm}
\begin{table}[ht]
	\begin{center}
	{\small 
	\begin{tabular}{|c||c|c|c|c||c|c|c|c|}\hline
	&\multicolumn{4}{c||}{$\gam=(0,0)^\top$}&\multicolumn{4}{|c|}{$\gam=\gam^\star=(0.031521,-0.005534)^{\!\top}$} \\\hline \hline
	$\ell$     &  $L^2$ &  rate     & $L^2_\alpha$   &    rate &  $L^2$ &  rate     & $L^2_{\alpha}$   &    rate     \\ \hline
	2    & 7.0709e--03&    -     & 1.9465e--03&    -    &5.6222e--03 &   -    &   1.8330e--03&   -     	  \\
	3    &	2.7238e--03&  1.38  & 6.1921e--04&  1.65&1.5194e--03 & 1.89&	1.8777e--04&  3.29  \\
	4    & 1.0452e--03&  1.38 & 2.3725e--04&  1.38&4.4752e--04 & 1.76&	1.8210e--05&  3.37  \\
	5    &	4.0414e--04&  1.37 &  9.3743e--05&  1.34&1.3749e--04 & 1.70&	1.9942e--06&  3.19  \\
	6    &	1.5751e--04&  1.36 &  3.7180e--05&  1.33&4.2897e--05 & 1.68&	2.3592e--07&  3.08  \\ 
	7    &	6.1766e--05&  1.35 &  1.4753e--05&  1.33&1.3461e--05 & 1.67&	2.8883e--08&  3.03  \\ 
	8    &	2.4323e--05&  1.34 &  5.8546e--06&  1.33&4.2339e--06 & 1.67&	3.5837e--09&  3.01  \\  \hline 
	$\ell$     &  $H^1$ &  rate     & $H^1_{\alpha}$   &    rate  &  $H^1$ &  rate     & $H^1_{\alpha}$ &    rate     \\ \hline
	2    &		3.8901e--02&  -      &	    1.2469e--02&    -    & 4.3263e--02&    -   	& 1.6682e--02 &   -   	    	\\
	3    &  	2.4482e--02&  0.67	 &		3.6263e--03&  1.78& 2.4334e--02&  0.83	& 3.2480e--03 & 2.36		\\
	4    &  	1.5461e--02&  0.66	 &		1.2070e--03&  1.59& 1.4871e--02&  0.71	& 7.3695e--04 & 2.14		\\
	5    &  	9.7522e--03&  0.66	 &		4.4090e--04&  1.45& 9.2795e--03&  0.68	& 1.7799e--04 & 2.05		\\
	6    &  	6.1468e--03&  0.67	 &		1.6891e--04&  1.38& 5.8257e--03&  0.67	& 4.3872e--05 & 2.02		\\ 
	7	   &  	3.8731e--03&  0.67	 &		6.6080e--05&  1.35& 3.6651e--03&  0.67	& 1.0900e--05 & 2.01		\\
	8	   &  	2.4401e--03&  0.67	 &		2.6074e--05&  1.34& 2.3077e--03&  0.67	& 2.7175e--06 & 2.00		\\ \hline
	\end{tabular}}
	\label{tab: L_Shape_ring_rates}
	\caption{ $\omega = 3/2\pi$: Errors and convergence rates of the approximative solution and the gradients. Left columns: $u-u_h$, right columns: $u-\uhm$ in standard $L^2$-norm and weighted $L^2_\alpha$ with $\alpha=1.3333$. Layer-modification $c_h^R(\cdot,\cdot)$, $k=2$.}
	\end{center}
	\vspace*{-0.5cm}
\end{table}

As one can see, standard finite elements exhibit $2\lambda_1$ and $\lambda_1$ convergence rates for the $L^2$-error in the solution and the gradient, respectively. Furthermore, optimal rates $3$ and $2$ in  the weighted $L^2_\alpha$-norm can be achieved using energy-correction. Note also that a simple change of the error-norm in the standard scheme does not qualitatively increase the convergence rates of the approximation error $u-\uh$.
\begin{table}[ht]
	\begin{center}{\small
	\begin{tabular}{|c||c|c|c|c||c|c|c|c|}\hline
	&\multicolumn{4}{c||}{$\gam=(0,0)^\top$}&\multicolumn{4}{|c|}{$\gam=\gam^\star=(0.073696, -0.019675)^{\!\top}$} \\\hline \hline
	$\ell$     &  $L^2$ &  rate     & $L^2_\alpha$   &    rate &  $L^2$ &  rate     & $L^2_{\alpha}$   &    rate     \\ \hline
	2    & 1.4688e--02 &   -     	&  4.0231e--03 &   -    	& 	1.4058e--02 &   -   		&	5.5315e--03 &   -   			\\
	3    & 6.5058e--03 & 1.17 	&	1.5652e--03 & 1.36	&	3.5969e--03 & 1.97	&	6.4029e--04 & 3.11			\\
	4    & 2.8727e--03 & 1.18 	&	6.9520e--04 & 1.17	&	1.0173e--03 & 1.82	&	6.0880e--05 & 3.39			\\
	5    & 1.2781e--03 & 1.17 	&	3.1399e--04 & 1.15	&	3.1468e--04 & 1.69	&	6.0667e--06 & 3.33			\\
	6    & 5.7257e--04 & 1.16 	&	1.4210e--04 & 1.14	&	1.0189e--04 & 1.63	&	6.4267e--07 & 3.24			\\ 
	7    & 2.5768e--04 & 1.15 	&	6.4335e--05 & 1.14	&	3.3693e--05 & 1.60	&	7.1739e--08 & 3.16			\\ 
	8	  & 1.1628e--04 & 1.15  & 	2.9132e--05 & 1.14	&	1.1248e--05 & 1.58	&	8.3485e--09 & 3.10		\\ \hline	
	$\ell$     &  $H^1$ &  rate     & $H^1_{\alpha}$   &    rate  &  $H^1$ &  rate     & $H^1_{\alpha}$ &    rate     \\ \hline
	2& 7.2533e--02&    -        &   2.1898e--02&    -    & 	9.7586e--02&    -      & 	4.2072e--02&    -         \\
	3& 4.9699e--02&  0.55    &   7.7725e--03&  1.49&    5.3057e--02&  0.88  &   7.3131e--03&  2.52		\\ 
	4& 3.3713e--02&  0.56    &   3.1786e--03&  1.29&    3.3087e--02&  0.68  &   1.4913e--03&  2.29		\\
	5& 2.2769e--02&  0.57    &   1.3896e--03&  1.19&    2.1653e--02&  0.61  &   3.4514e--04&  2.11		\\
	6& 1.5347e--02&  0.57    &   6.2221e--04&  1.16&    1.4407e--02&  0.59  &   8.3736e--05&  2.04		\\
	7& 1.0335e--02&  0.57    &   2.8078e--04&  1.15&    9.6473e--03&  0.58  &   2.0675e--05&  2.02		\\
	8& 6.9573e--03&  0.57    &   1.2701e--04&  1.14&    6.4780e--03&  0.57  &   5.1410e--06&  2.01		\\ \hline
 \end{tabular}	}
\label{tab:angle_shape_ring_rates}
	\end{center}
	\caption{$\omega = 7/4\pi$:  Errors and convergence rates of the approximative solution and the gradients. Left columns: $u-u_h$, right columns: $u-\uhm$ in standard $L^2$-norm and weighted $L^2_\alpha$ with $\alpha=1.3333$. Layer-modification $c_h^R(\cdot,\cdot)$, $k=2$.}
		\vspace*{-0.5cm}
\end{table}

\subsubsection{Mesh requirements: $i\neq j$ pairing in \cref{aprconvstr2}}
 The proof of \cref{l:aprconv2} is based on the \cref{U}-assumption.
Here we study the mesh properties in more detail. Firstly we observe that the full symmetry (G3) is a too strong assumption for \cref{assint}, and it can be relaxed to more general cases. For example, classical criss-cross meshes, see the case for $\omega=2\pi$ in \cref{fig:geoms1}, satisfy the interpolation property \cref{assint} as well. In \cref{tab:u3ortho} we show $\int_{\Upsilon} \nabla s_1\cdot \nabla I^2_h s_3 dx$ on different refinement levels, with $\Upsilon$-domains as depicted on the left of the tables.

\begin{table}[h!]
	\begin{tabular}{c}
		\includegraphics[width=0.17\textwidth]{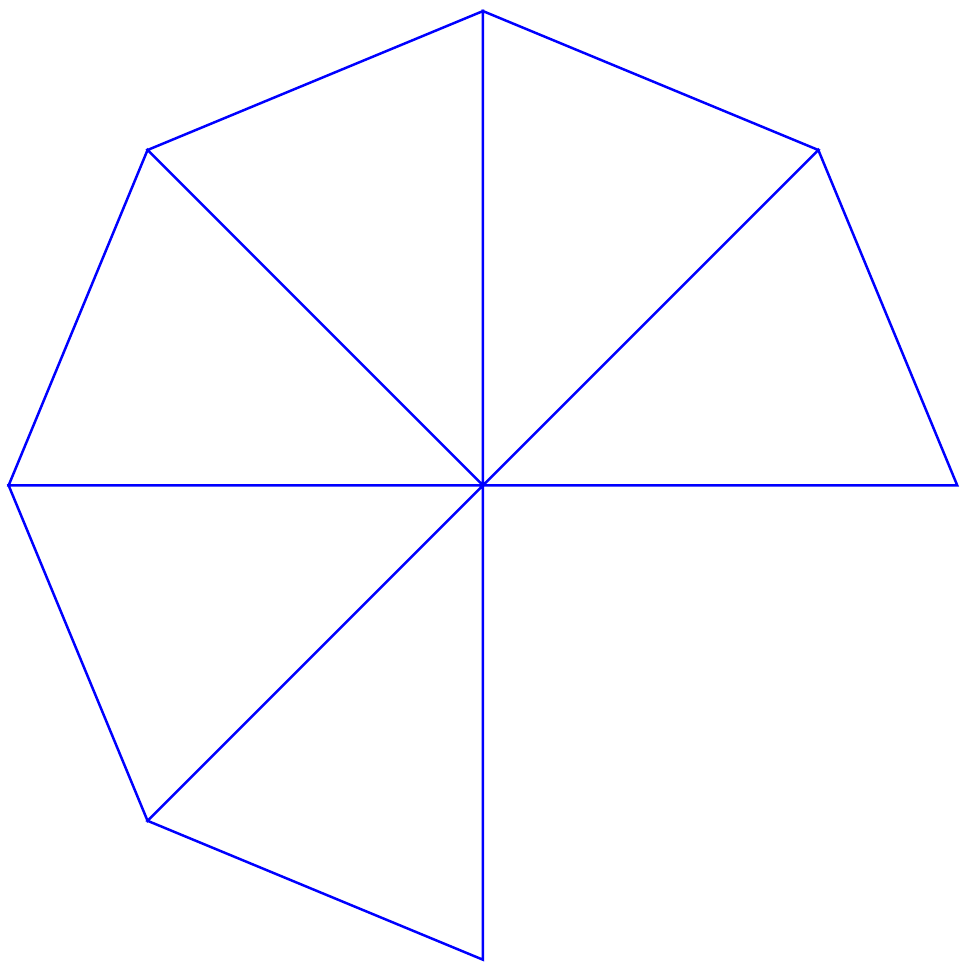}
	\end{tabular}
	\hspace*{-0.2cm}
	\begin{tabular}{|c||c|}\hline 
	$\ell$ & $\Upsilon:$ (G3) mesh      \\ \hline
	1   &    $-$2.636780e--15 \\
	2   &    $-$3.774758e--15 \\
	3   &    $-$3.734860e--15 \\
	4   &    $-$3.631210e--15 \\
	5   &    $-$3.642052e--15 \\ \hline
	\end{tabular}
	\hspace*{0.7cm}
	\begin{tabular}{c}
		\includegraphics[width=0.17\textwidth]{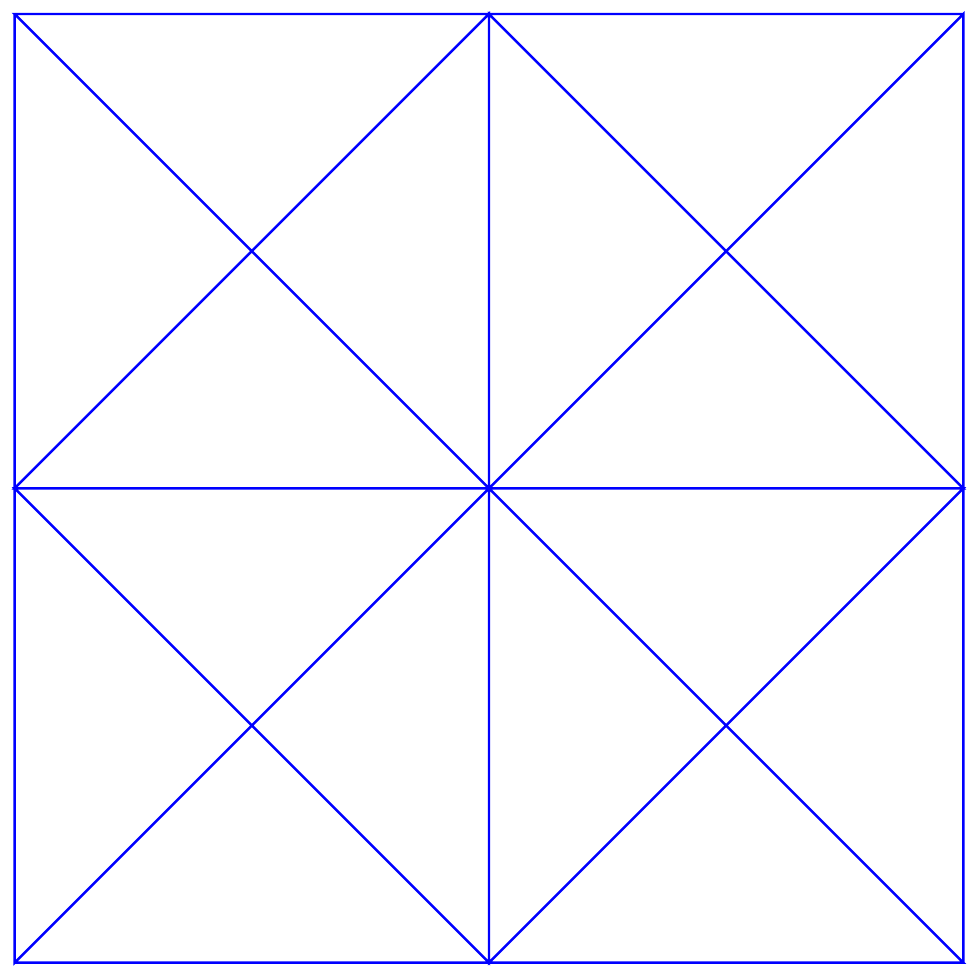}
	\end{tabular}
	\hspace*{-0.2cm}
	\begin{tabular}{|c||c|}\hline
	$\ell $& $\Upsilon:$  criss-cross   \\ \hline  
	1 &    $-$2.636780e--15    \\
	2 &    $-$3.774758e--15    \\
	3 &    $-$3.734860e--15    \\
	4 &    $-$3.631210e--15    \\
	5 &    $-$3.642052e--15    \\ \hline
	\end{tabular}
	\vspace*{0.1cm}
	\label{tab:u3ortho}
		\caption{Values of $\int_{\Upsilon} \nabla s_1\cdot \nabla I^2_h s_3 \,dx$ for the depicted meshes for different refinement levels.}
\end{table}

Secondly, we illustrate the role of the mesh symmetry. Hence, we perform two types of computations, first, using a fully non-symmetric mesh, and a mesh which is locally mirror-symmetric but not satisfying the (G3)-property, see the meshes on the left side of \cref{tab:nonuni_pairs}, for the case of $k=2$ and $\omega=7/4\pi>4/3\pi$. Thus for this combination of $k$ and $\omega$ the \cref{U}-property of the mesh is not satisfied. Again, since we consider the second order approximations, we require for the optimal convergence rates of the solution, i.e., $\|u-\uhm\|_{0,\alpha} = \mathcal{O}(h^3)$, the convergence rate in $\tilde g_h(s_i,s_j)$ to be also $O(h^3)$ for both pairing, $(s_1, s_2)$ and $(s_1, s_3)$, see \cref{F:dist}.

\begin{table}[ht]
	\begin{center}
   \begin{tabular}{c}
		\includegraphics[width=0.2\textwidth]{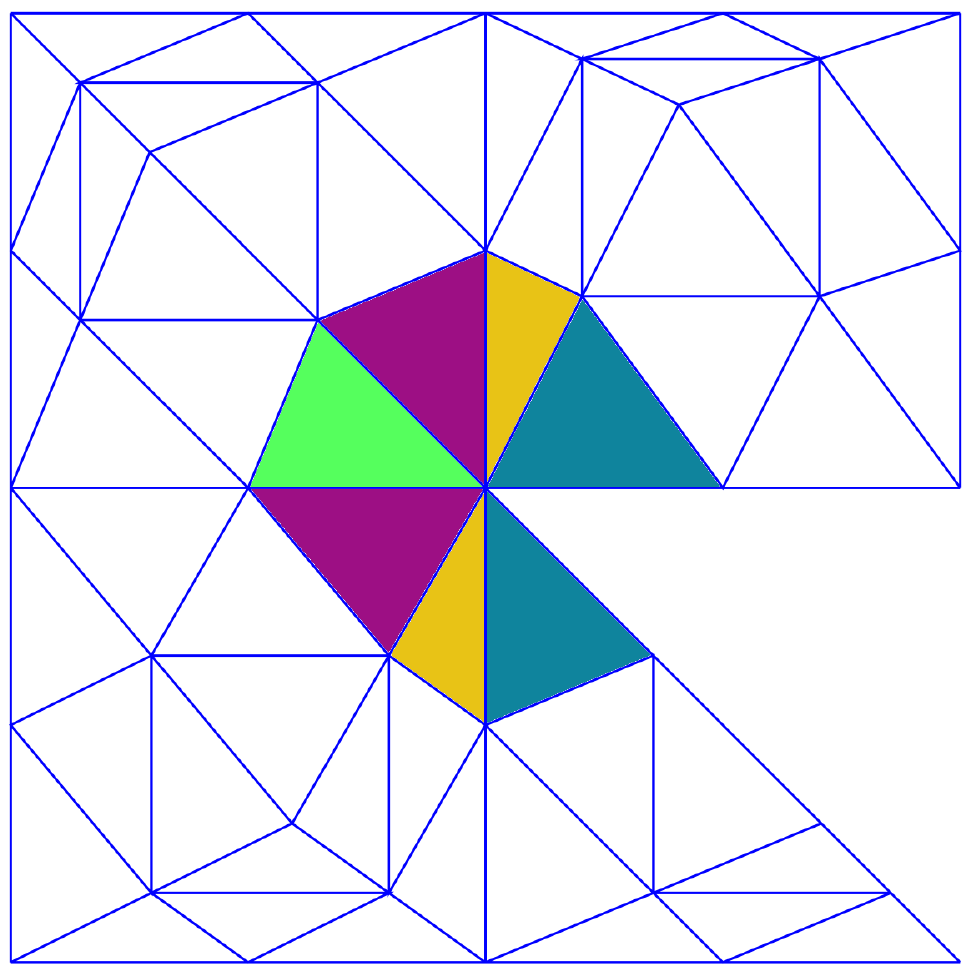}
	\end{tabular}
	\hspace*{0.8cm}
	\begin{tabular}{|c||c|c|c|c|}\hline
	&\multicolumn{4}{c|}{(G1) non-symmetric} \\ \hline
	$\ell$ & $\tilde{g}_h(s_1,s_2)$   &    rate &  $\tilde{g}_h(s_1,s_3)$ &  rate   \\ \hline
	2& 	   1.7836e--04&   --  &	1.9118e--04&   --  								\\
	3&	   5.3969e--05&  1.72	&	3.8239e--05&  2.32 						\\
	4&	   1.6437e--05&  1.72	&	7.8252e--06&  2.29 						\\
	5&	   5.0081e--06&  1.71	&	1.6040e--06&  2.29 						\\
	6&	   1.5260e--06&  1.71 &	3.2887e--07&  2.29 					\\
	7&	   4.6503e--07&  1.71	&	6.7436e--08&  2.29 						\\
	8&     1.4172e--07&  1.71 &	1.3829e--08&  2.29 					\\ \hline
	\end{tabular}\\[1em]
  \begin{tabular}{c}
		\includegraphics[width=0.2\textwidth]{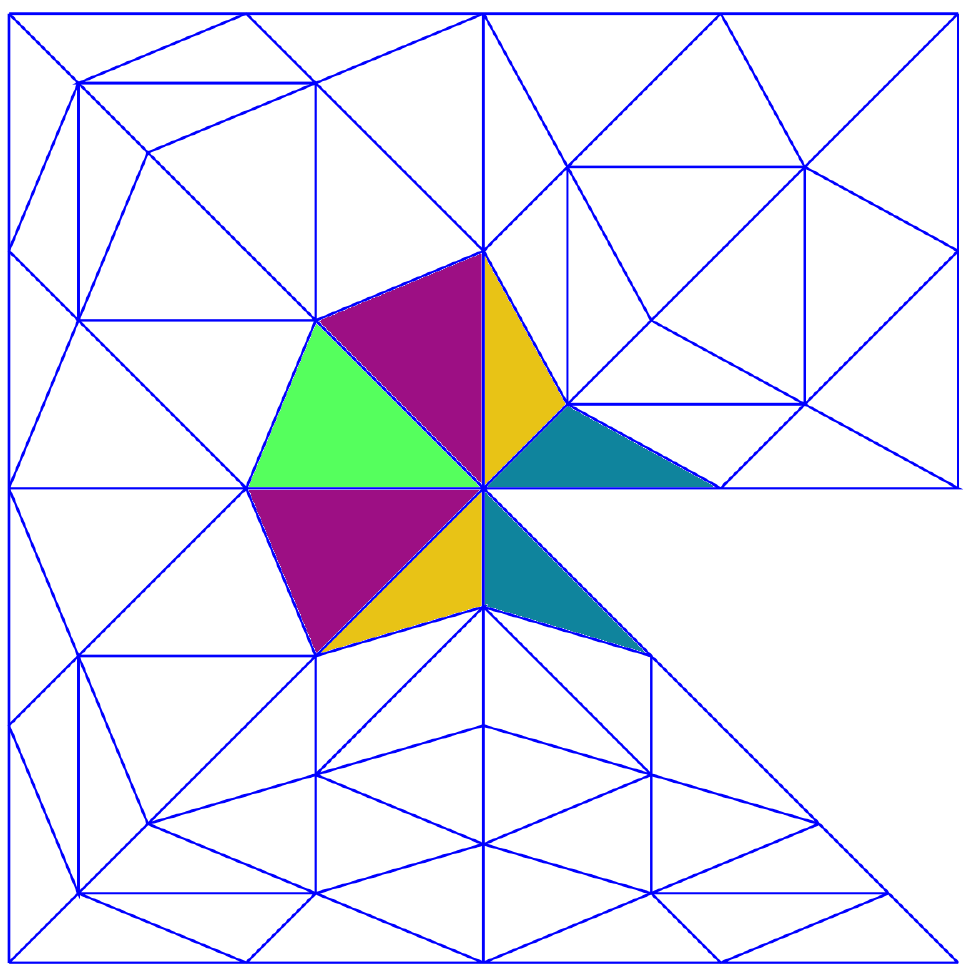}
	\end{tabular}	
	\hspace*{0.8cm}
		\begin{tabular}{|c||c|c|c|c|}\hline
	&\multicolumn{4}{c|}{(G2) mirror-symmetric} \\ \hline
	$\ell$ &  $\tilde{g}_h(s_1,s_2)$   &    rate &  $\tilde{g}_h(s_1,s_3)$ &  rate       \\ \hline
	2& 	    3.4985e--06&  --  	 & 1.3550e--04&  --    \\
	3&	  	3.6461e--07&  3.26	 & 2.8173e--05&  2.27  \\
	4&	  	2.3000e--08&  3.99	 & 5.7796e--06&  2.29  \\
	5&	  	1.4154e--09&  4.02	 & 1.1833e--06&  2.29  \\
	6&	    8.7613e--11&  4.01	 & 2.4236e--07&  2.29  \\
	7&	    5.4420e--12&  4.01	 & 4.9666e--08&  2.29  \\
	8&      3.5838e--13&  3.92	 & 1.0181e--08&  2.29  \\ \hline
	\end{tabular}
	\end{center}
	\label{tab:nonuni_pairs} 
			\caption{Two examples of $\omega = 7/4\pi$ meshes which do not satisfy (G3), but the property (G1) for the first row and the property (G2) for the second row. }
\end{table}

In \cref{tab:nonuni_pairs}, we present the numerical results  computed on the above mentioned meshes demonstrating that (G1) or (G2) properties of the local mesh are not sufficient. As one can see in the case of a locally mirror-symmetric mesh, the necessary condition on $\tilde g_h(s_i,s_j)$ is violated only for the $(s_1,s_3)$ pairing, while the $(s_1,s_2)$ pairing has even higher convergence order than we require. This comes from the fact, that $s_1$ is symmetric and $s_2$ is anti-symmetric in the $\theta$-direction, and thus, the integral of the product of their gradients on mirror-symmetric mesh vanishes. The violation of the local symmetry of the mesh then results that also this pairing does not exhibit the required convergence rate.

\subsection{Layer-modification $c_h^R(\cdot,\cdot)$: $k=3$, $k=4$}\label{ss:thirdfourth}
In this section, we present results for the higher order methods, i.e., the third and fourth approximations, and concentrate on two re-entrant corner cases: $3/2\pi$ (\cref{tab:3pi2k34}) and  $2\pi$ (\cref{tab:2pi1k34}).

\begin{table}[h]
	\begin{center}{\small
	\begin{tabular}{|c||c|c|c|c||c|c|c|c|}\hline
	&\multicolumn{4}{c||}{$k = 3$}&\multicolumn{4}{|c|}{$k = 4$} \\\hline \hline
	$\ell$&  $\| u-u_h \|_{0,\alpha}\!\!$ &  rate  & $\| u-\uhm \|_{0,\alpha}\!\!$ & rate & $\| u-u_h \|_{0,\alpha}\!\!$ & rate & $\| u-\uhm \|_{0,\alpha}\!\!$ & rate \\ \hline
	2    &    3.6458e--04&  -    &4.7531e--04&  -		& 1.4972e--04&  -    & 	4.3849e--04 &  -   	\\
	3    &	  1.4914e--04&  1.29 &3.5013e--05&  3.76& 5.6715e--05&  1.40 & 	1.4127e--05 &  4.96	\\
	4    &  	5.9232e--05&  1.33 &1.9895e--06&  4.14& 2.2495e--05&  1.33 & 	2.9706e--07 &  5.57	\\
	5    &		2.3504e--05&  1.33 &1.1977e--07&  4.05& 8.9263e--06&  1.33 & 	7.2456e--09 &  5.36	\\
	6    &	  9.3270e--06&  1.33 &7.4012e--09&  4.02& 3.5423e--06&  1.33 & 	1.9010e--10 &  5.25	\\ 
	7    &	 	3.7013e--06&  1.33 &4.5990e--10&  4.01&            &       &              &      	\\  \hline
	\end{tabular}                                                           }
	\label{tab:3pi2k34}
		\caption{$\omega = 3/2\pi$: Errors and convergence rates with ($\uhm$) and without energy-correction ($u_h$), for $k=3,4$, and $\alpha_3=2.3333$, $\alpha_4=3.3333$. Layer-modification $c_h^R(\cdot,\cdot)$ with correction parameters for third ($\gam_3$) and fourth ($\gam_4$) order.\newline
	$\gam_3=(0.0128913716057252,-0.00236670558729706)^\top$\newline
	\mbox{$\gam_4=(0.00770287127448,-0.00247845692351,0.000451799729448)^\top$.}}
\end{center}
\end{table}
\begin{table}[ht]
\begin{center}{\small
	\begin{tabular}{|c||c|c|c|c||c|c|c|c|}\hline
	&\multicolumn{4}{c||}{$k = 3$}&\multicolumn{4}{|c|}{$k = 4$} \\\hline \hline
	$\ell$&  $\| u-u_h \|_{0,\alpha_3}\!\!\!\!$ &  rate  & $\| u-\uhm \|_{0,\alpha_3}\!\!\!\!$ & rate & $\| u-u_h \|_{0,\alpha_4}\!\!\!\!$ & rate & $\| u-\uhm \|_{0,\alpha_4}\!\!\!\!$ & rate \\ \hline
	2    &	1.6938e--03 & -    	&	  7.3218e--03 & -   	&	7.8234e--04 &  -  &	8.2906e--03 & -    \\
	3    &	8.4536e--04 & 1.00 	&	  4.0724e--04 & 4.17	&	3.8899e--04 & 1.01&	5.0596e--04 & 4.03 \\
	4    &	4.2165e--04 & 1.00 	&	  1.9310e--05 & 4.40	&	1.9423e--04 & 1.00&	9.1646e--06 & 5.79 \\
	5    &	2.1059e--04 & 1.00 	&	  1.0565e--06 & 4.19	&	9.7052e--05 & 1.00&	2.0680e--07 & 5.47 \\
	6    &	1.0524e--04 & 1.00 	&	  6.2341e--08 & 4.08	&	4.8510e--05 & 1.00&	6.2929e--09 & 5.04 \\
	7    &	5.2604e--05 & 1.00 	&	  3.7994e--09 & 4.04	&             &     &             &      \\ \hline
	\end{tabular}}
	\label{tab:2pi1k34}
	\caption{$\omega = 2\pi$: Errors and convergence rates with ($\uhm$) and without energy-correction ($u_h$), for $k=3,k=4$, and $\alpha_3=2.5$, $\alpha_4=3.5$. Layer-modification $c_h^R(\cdot,\cdot)$ with correction parameters for third ($\gam_3$) and fourth ($\gam_4$) order. \newline
	$\gam_3=(0.0837702755930500,-0.0571791329689500,0.01578275592125006)^\top$\newline
	$\gam_4=(0.0593735765703, -0.0603253013792,0.0339890181456,$ $-0.00836417112629)^\top$.}
	\end{center}
\end{table}
\indent In all cases, the optimal rates are obtained in the $L^2_{\alpha}$ norm if a sufficient number of correction parameters is used. Such can be read out from \cref{F:dist}. Note also, that the weight $\alpha$ increases with the approximation order $k$, namely $\alpha =  k-\lambda_1$.

\subsection{Function-modification $c_h^F(\cdot,\cdot)$}\label{ss:chF}

As described in \cref{S:gamma}, higher order approximations require an increased number of correction parameters $\gamma_i$, $i=1,\dots,K\leqslant k$. This means that a~priori, the correction by layers $c_h^R(\cdot,\cdot)$ needs a larger support, constructed by $K$-layers of elements around the singular point. To have a possibility to avoid this fact, we have suggested a correction by a function, $c_h^F(\cdot,\cdot)$, see \cref{chF}. Its support is always restricted to a single element layer at the singular points.

We present the asymptotic correction-function $\sum_{i=1}^K \, {\gamma_i^F}^\star \, \hat{r}^{i-1}$ for all the considered mesh and approximation-order settings in \cref{fig:gammas}. Note here the bound of the maximum of the correction-function, guaranteeing the ellipticity of $a_h(\cdot,\cdot)$.

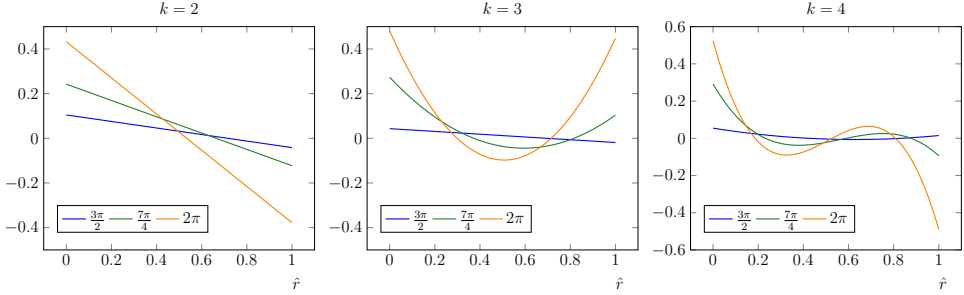
\begin{figure}[]
\begin{center}
\newcommand{\sclf}{0.52}
\begin{tikzpicture}[scale = \sclf,font = \Large]
\definecolor{darkgreen}{rgb}{0.125,0.5,0.169}
	\begin{axis}[
		title={ $k=2$ },
    xlabel style={at={(0.9,-0.1)},anchor=north west},
		xlabel={$\hat{r}$},
		legend style={at={(0.05,0.2)},anchor=north west},legend columns=4,legend cell align=left,
		ymin=-0.5, ymax=0.5,
	]
	\addplot[domain=0:1,color=blue,thick,samples=100] {0.104629368522 -0.146333651229*x};   
	\addplot[domain=0:1,color=darkgreen,thick,samples=100] {0.242541884302 -0.365466359337*x};   
	\addplot[domain=0:1,color=orange,thick,samples=100] {0.432790956361 -0.810825080057*x};   
	\legend{$\frac{3\pi}{2}$,$\frac{7\pi}{4}$,$2\pi$}
	\end{axis}
\end{tikzpicture}
\begin{tikzpicture}[scale = \sclf,font = \Large]
	\begin{axis}[
		title={ $k=3$ },
    xlabel style={at={(0.9,-0.1)},anchor=north west},
		xlabel={$\hat{r}$},
		legend style={at={(0.05,0.2)},anchor=north west},legend columns=4,legend cell align=left,
		ymin=-0.5, ymax=0.5,
	]
	\addplot[domain=0:1,color=blue,thick,samples=100] {0.0430940204369 -0.0619934913075*x};   
	\addplot[domain=0:1,color=darkgreen,thick,samples=100] {0.27331111412 -1.06868776126*x + 0.899221652837*x^2};   
	\addplot[domain=0:1,color=orange,thick,samples=100] {0.481131458404 -2.28146684246*x + 2.24881531841*x^2};   
	\legend{$\frac{3\pi}{2}$,$\frac{7\pi}{4}$,$2\pi$}
	\end{axis}
\end{tikzpicture}
\begin{tikzpicture}[scale = \sclf,font = \Large]
	\begin{axis}[
		title={ $k=4$ },
    xlabel style={at={(0.9,-0.1)},anchor=north west},
		xlabel={$\hat{r}$},
		legend style={at={(0.05,0.2)},anchor=north west},legend columns=4,legend cell align=left,
		ymin=-0.6, ymax=0.6,
	]
	\addplot[domain=0:1,color=blue,thick,samples=100] {0.0545692518571 -0.194380081181*x + 0.154642913842*x^2};   
	\addplot[domain=0:1,color=darkgreen,thick,samples=100] {0.291964004854 -2.08641092234*x + 4.1327001074*x^2 -2.43279768432*x^3};   
	\addplot[domain=0:1,color=orange,thick,samples=100] {0.5244101444958560 -4.4685139512732723*x + 10.0848142531165568*x^2 -6.6307445001949796*x^3};   
	\legend{$\frac{3\pi}{2}$,$\frac{7\pi}{4}$,$2\pi$}
	\end{axis}
\end{tikzpicture}
\end{center}
\caption{Asymptotic correction functions $\sum_{i=1}^K \, {\gamma_i^F}^\star \, \hat{r}^{i-1}$ for three different approximation orders and three different angles. These function-corrections were used for the computations presented in \cref{fig:comp}.}
\label{fig:gammas}
\end{figure}

\begin{figure}[h!]
\begin{center}
\begin{tabular}{cc}
\begin{tikzpicture}[scale = 0.7]
		\begin{loglogaxis}[
		 title={$k=2$},
			xlabel={DoFs}, xlabel style={at={(0.9,-0.1)},anchor=north west},
			ylabel={error},
			legend pos=south west
		]			
		\addplot[color = gray, densely dashed, thick] coordinates {
		(     33, 0.006695294459835)
		(394241 , 0.000000005127408)
		};
		\addplot[mark=o, color = green] coordinates {
		( 113   ,    1.946491e-03)
		( 417   ,    6.192110e-04)
		( 1601  ,    2.372484e-04)
		( 6273  ,    9.374264e-05)
		( 24833  , 	 3.717952e-05)
		( 98817  , 	 1.475306e-05)
		( 394241  ,  5.854592e-06)
		};
		\addplot[mark=o, color = red] coordinates {
		( 113   ,   1.832986e-03 )
		( 417   ,   1.877747e-04 )
		( 1601  ,   1.820988e-05 )
		( 6273  ,   1.994188e-06 )
		( 24833 ,   2.359184e-07 )
		( 98817 ,   2.888256e-08 )
		( 394241,   3.583656e-09 )
		};

		\addplot[mark=o, color = blue] coordinates {
		( 33    ,  6.249065e-03     )
		( 113   , 6.625334e-04     )
		( 417   ,7.277360e-05      )
		( 1601  ,8.518649e-06      )
		( 6273  ,1.035930e-06      )
		( 24833  ,1.281105e-07 	  )
		( 98817  ,1.594813e-08 	  )
		};

		\addplot[mark=triangle, color = green] coordinates {
		(133     , 	4.023119e-03 	)
		( 489    ,	   1.565214e-03 		)
		( 1873    ,	 6.951998e-04	)
		( 7329     , 3.139851e-04		)
		( 28993	  ,	 1.420978e-04	)
		( 115329,	 6.433457e-05  	)
		( 460033, 	 2.913151e-05	)
		};

		\addplot[mark=triangle, color = red] coordinates {
		(133     , 	5.531486e-03	)
		( 489    ,	6.402928e-04		)
		( 1873    ,	6.088025e-05	)
		( 7329     ,6.066699e-06		)
		( 28993	  ,	6.426752e-07	)
		( 115329,	7.173956e-08  	)
		( 460033, 	8.349078e-09	)
		};

		\addplot[mark=triangle, color = blue] coordinates {
		(39      , 		1.329756e-02 	)
		(133     , 		1.440387e-03 )
		( 489    ,		1.505234e-04 	)
		( 1873    ,		1.524973e-05 )
		( 7329     ,	1.577991e-06 	)
		( 28993	  ,		1.703569e-07 )
		( 115329,	  	1.924132e-08 )
		};

		\addplot[mark=square, color = green] coordinates {
		( 149,		5.447541e-03				)
		( 553,		2.555285e-03				)
		(2129,		1.266709e-03		 		)
		( 8353,		6.318142e-04				)
		( 33089,	    3.155603e-04					)
		( 131713,	1.576946e-04				)
		(525569	,	7.882600e-05				)
		};

		\addplot[mark=square, color = red] coordinates {
		( 149,			7.484737e-03			)
		( 553,			9.613364e-04			)
		(2129,			1.071559e-04	 		)
		( 8353,			1.253320e-05			)
		( 33089,		    1.522378e-06				)
		( 131713,		1.880485e-07			)
		(525569	,		2.339125e-08			)
		};

		\addplot[mark=square, color = blue] coordinates {
		(43,			1.765647e-02	 	)
		( 149,			2.255566e-03	)
		( 553		,	2.844502e-04		)
		(2129,			3.504530e-05	 )
		( 8353,			4.324484e-06	)
		( 33089,		5.362578e-07		)
		( 131713,		6.673934e-08	)
		};		
		\legend{  $\mathcal{O}(h^{3})$}
		\end{loglogaxis}
	\end{tikzpicture}
\hspace*{-0.5cm}
&
\begin{tikzpicture}[scale = 0.7]
		\begin{loglogaxis}[
		    title={$k=3$},
			xlabel={DoFs}, xlabel style={at={(0.9,-0.1)},anchor=north west},
			ylabel={error},
			legend pos=south west
		]
		\addplot[color = gray, densely dashed, thick] coordinates {
		(     67,  0.829187396351576)
		(221953 , 0.000000075557966)
		};
		\addplot[mark=o, color = green] coordinates {
		(241,  	 3.645809e-04)
		(913, 	 1.491442e-04)
		(3553, 	 5.923211e-05)
		(14017,  2.350386e-05)
		(55681,  9.327001e-06)
		(221953, 3.701337e-06)
		};	
		\addplot[mark=o, color = red] coordinates {
		(241,  	4.753118e-04)
		(913, 	 3.501259e-05)
		(3553, 	 1.989514e-06)
		(14017,  1.197730e-07)
		(55681,   7.401174e-09)
		(221953,  4.599012e-10)
		};
		\addplot[mark=o, color = blue] coordinates {
		(67,     2.265797e-03)
		(241,    1.394305e-04)
		(913,    8.586574e-06)
		(3553,   5.332193e-07)
		(14017,  3.336425e-08)
		(55681,  2.094811e-09)
		};
		\addplot[mark=triangle, color = green] coordinates {
     (283, 9.761375e-04)
		(1069, 4.526332e-04)
		(4153, 2.048523e-04)
		(16369, 9.272159e-05)
		(64993, 4.198005e-05)
		(259009, 1.900910e-05)
		};			
		\addplot[mark=triangle, color = red] coordinates {
		(283,		 3.893841e-03	)
		(1069,		 2.450379e-04	)
		(4153,		 9.673994e-06	)
		( 16369,	 4.490788e-07		)
		( 64993, 	 2.358691e-08		)
		( 259009,  	 1.345556e-09		)
		};
		\addplot[mark=triangle, color = blue] coordinates {
		( 79,      5.218032e-03         )
		(283,		 2.726081e-04)
		(1069,		 1.586614e-05)
		(4153,		 9.654047e-07)
		( 16369,	 5.958160e-08	)
		( 64993, 	 3.695316e-09	)
		};				
		\addplot[mark=square, color = green] coordinates {
		(319,		1.693843e-03 	)
		(1213,		8.453612e-04 	)
		(4729,		4.216530e-04 	)
		( 18673,	    2.105887e-04 		)
		(74209 , 	1.052363e-04 		)
		(295873,   	5.260370e-05 		)
		};			
		\addplot[mark=square, color = red] coordinates {
		(319,		7.321811e-03	)
		(1213,		4.072352e-04	)
		(4729,		1.930997e-05	)
		( 18673,	1.056526e-06		)
		(74209 , 	6.234117e-08		)
		(295873,   	3.799362e-09		)
		};	
		\addplot[mark=square, color = blue] coordinates {
		(88,    9.584604e-03    )
		(319,	5.044586e-04		)
		(1213,	2.665950e-05		)
		(4729,	1.522315e-06		)
		( 18673,9.077171e-08			)
		(74209 ,5.524485e-09 			)
		};		
		\legend{ $\mathcal{O}(h^{4})$}
		\end{loglogaxis}
	\end{tikzpicture}
\\[-0.8em]
\begin{tikzpicture}[scale = 0.7]
		\begin{loglogaxis}[
		 title={$k=4$},
			xlabel style={at={(0.9,-0.1)},anchor=north west},
			xlabel={DoFs},
			ylabel={error},
			legend style={at={(0,-0.1)},anchor=north west},legend columns=4,
		]
		\addplot[color = white] coordinates {(1000,0.0001)};
		\addplot[mark=o, color = green] coordinates {
		(417,    1.497193e-04 	)
		(1601,	 5.671462e-05   		)
		(6273,  2.249465e-05   	)
		(24833,  8.926291e-06  	)
		(98817, 3.542297e-06    )
		};
		\addplot[mark=triangle, color = green] coordinates {
		(489,			4.334723e-04		)
		(1873,			1.921407e-04   		)
		(7329,			8.694603e-05	  	)
		(28993,			3.936280e-05	  	)
		(115329,		1.782351e-05		    )
		};
		\addplot[mark=square, color = green] coordinates {
		(553	,		7.823369e-04 		)
		(2129 ,			3.889854e-04    		)
		(8353 ,			1.942339e-04 	  	)
		(33089 ,		9.705246e-05 		  	)
		(131713 ,		4.851014e-05 	    )
		};
		\addplot[color = white] coordinates {(1000,0.0001)};
		\addplot[mark=o, color = red] coordinates {
		(417,     4.384902e-04    	 	)
		(1601,		   1.412671e-05	)
		(6273,   2.970643e-07	 		)
		(24833,  7.245611e-09	 		)
		(98817,  1.900967e-10	  		)
		};
		\addplot[mark=triangle, color = red] coordinates {
		(489,		3.589537e-03 			)
		(1873,		2.521685e-04 	   		)
		(7329,		3.853314e-06 		  	)
		(28993,		8.098833e-08 		  	)
		(115329,	 2.114137e-09			    )
		};
		\addplot[mark=square, color = red] coordinates {
		(553	,		8.290572e-03		)
		(2129 ,			5.059554e-04   		)
		(8353 ,			9.164586e-06	  	)
		(33089 ,		2.067993e-07		  	)
		(131713 ,		6.292925e-09	    )
		};
		\addplot[color = white] coordinates {(1000,0.0001)};
		\addplot[mark=o, color = blue] coordinates {
		(113,    1.185386e-03 				    )
		(417,    3.296624e-05 					)
		(1601,	 1.011558e-06				)
		(6273,   3.147185e-08				  	)
		(24833,	 9.818780e-10			   	)
		};
		\addplot[mark=triangle, color = blue] coordinates {
		(133,		2.240668e-03				)
		(489,		4.861980e-05			)
		(1873,		1.286465e-06	   		)
		(7329,		3.731360e-08		  	)
		(28993,		1.110847e-09		  	)
		};
		\addplot[mark=square, color = blue] coordinates {
		(149 , 		4.135515e-03		)
		(553	,	1.067938e-04			)
		(2129 ,		2.639598e-06	   		)
		(8353 ,		8.117869e-08		  	)
		(33089 ,	3.048464e-09			  	)
		};
		\addplot[color = gray, densely dashed, thick, forget plot] coordinates {
		(113     , 0.594640245501642       )
		(98817 ,    0.000000026294883  )
		};\label{oh}
		\legend{\hspace*{-0.5cm}none:, $\frac{3\pi}{2}$,$\frac{7\pi}{4}$,$2\pi$,\hspace*{-0.5cm}$c_h^R:$,$\frac{3\pi}{2}$,$\frac{7\pi}{4}$,$2\pi$,\hspace*{-0.5cm}$c_h^F:$,$\frac{3\pi}{2}$,$\frac{7\pi}{4}$,$2\pi$}
		\end{loglogaxis}
		\node [draw,fill=white,anchor=south west] at (0.2,0.2) {\shortstack[l]{\small{\ref{oh} $\mathcal{O}(h^5)$}}};
	\end{tikzpicture}
\hspace*{-0.5cm}
&
\begin{tikzpicture}[scale = 0.7]
	\begin{semilogyaxis}[
		title={ $k=2,3,4$ },
    xlabel style={at={(0.9,-0.1)},anchor=north west},
		xlabel={$\ell$},
		ylabel={$\text{error}^{\text{R}}/\text{error}^{\text{F}}$},
		legend style={at={(0,-0.1)},anchor=north west},legend columns=4,legend cell align=left,
		ymin=0.5, ymax=1000,
	]	
	\addplot[color = white] coordinates {(3,10)};
	\addplot[mark=o, color = purple] coordinates {
	(1,		 4.384902e-04/   3.296624e-05 	)
	(2,      1.412671e-05/  1.011558e-06		) 
	(3, 	 2.970643e-07/	 3.147185e-08	 ) 
	(4, 	 7.245611e-09/	 9.818780e-10	)
  };		
	\addplot[mark=triangle, color = purple] coordinates {
	(1,	3.589537e-03/	4.861980e-05			 )
	(2, 2.521685e-04/	1.286465e-06			) 
	(3, 3.853314e-06/	3.731360e-08			  ) 
	(4, 8.098833e-08/	1.110847e-09				)
  };		
	\addplot[mark=square, color = purple] coordinates {
	(1,		8.290572e-03/	1.067938e-04			 )
	(2, 	5.059554e-04/	2.639598e-06			) 
	(3, 	 9.164586e-06/	8.117869e-08			  ) 
	(4, 	2.067993e-07/	3.048464e-09)
  };	
	\addplot[color = white] coordinates {(3,10)};
	\addplot[mark=o, color = black] coordinates {
	(1,	 		4.753118e-04/  		   1.394305e-04				)
	(2, 	    3.501259e-05/	  	   8.586574e-06		) 
	(3, 	    1.989514e-06/  		   5.332193e-07	 			) 
	(4, 	    1.197730e-07/ 	   3.336425e-08			)
	(5, 	     7.401174e-09/	   2.094811e-09	 		) 
  };	 
	\addplot[mark=triangle, color = black] coordinates {
	(1,		3.893841e-03/	2.726081e-04	   	)
	(2, 	2.450379e-04/	1.586614e-05		) 
	(3, 	9.673994e-06/  	9.654047e-07		) 
	(4, 	4.490788e-07/	5.958160e-08		)
	(5, 	2.358691e-08/	3.695316e-09		) 
   };	 
	\addplot[mark=square, color = black] coordinates {
	(1,		7.321811e-03/ 5.044586e-04			 )
	(2, 	4.072352e-04/	 2.665950e-05	) 
	(3, 	1.930997e-05/ 1.522315e-06  			) 
	(4, 	1.056526e-06/	 9.077171e-08	)
	(5, 	6.234117e-08/	 5.524485e-09	) 
   };	    
	\addplot[color = white] coordinates {(3,10)};
	\addplot[mark=o, color = orange] coordinates {
	(1,	 	1.832986e-03/ 	  6.625334e-04 	  			 )
	(2, 	1.877747e-04/	 7.277360e-05  	   			) 
	(3, 	1.820988e-05/ 	 8.518649e-06  	   			) 
	(4,		1.994188e-06/ 	 1.035930e-06  	   			 )
	(5, 	2.359184e-07/	 1.281105e-07				) 
	(6, 	2.888256e-08/	 1.594813e-08			) 
   };	
	\addplot[mark=triangle, color = orange] coordinates {
	(1,	 	5.531486e-03/	1.440387e-03		 )
	(2, 	6.402928e-04/ 	1.505234e-04		) 
	(3, 	6.088025e-05/ 	1.524973e-05		) 
	(4,		6.066699e-06/ 	1.577991e-06		 )
	(5, 	6.426752e-07/	1.703569e-07 		) 
	(6, 	7.173956e-08/ 	1.924132e-08		) 
   };
	\addplot[mark=square, color = orange] coordinates {
	(1,	       	7.484737e-03/ 2.255566e-03			 )
	(2, 	   	9.613364e-04/	 2.844502e-04	) 
	(3, 	   	1.071559e-04/	 3.504530e-05	) 
	(4,		   	1.253320e-05/	 4.324484e-06	 )
	(5, 	   	1.522378e-06/	 5.362578e-07     ) 
	(6, 	   	1.880485e-07/	 6.673934e-08	) 
   };	    
	\legend{\hspace*{-0.5cm}$k=4:$, $\frac{3\pi}{2}$,$\frac{7\pi}{4}$,$2\pi$,\hspace*{-0.5cm}$k=3:$,$\frac{3\pi}{2}$,$\frac{7\pi}{4}$,$2\pi$,\hspace*{-0.5cm}$k=2:$,$\frac{3\pi}{2}$,$\frac{7\pi}{4}$,$2\pi$}
	\end{semilogyaxis}
\end{tikzpicture}
\end{tabular}
\end{center}
\caption{Comparison of the standard and modified finite elements with energy corrections $c_h^R(\cdot,\cdot)$ and $c_h^F(\cdot,\cdot)$ for approximation orders $k=2,3,4$. The bottom right plot represents the ratio between the errors produced by the layer-correction and function-correction.}
\label{fig:comp}
\end{figure}
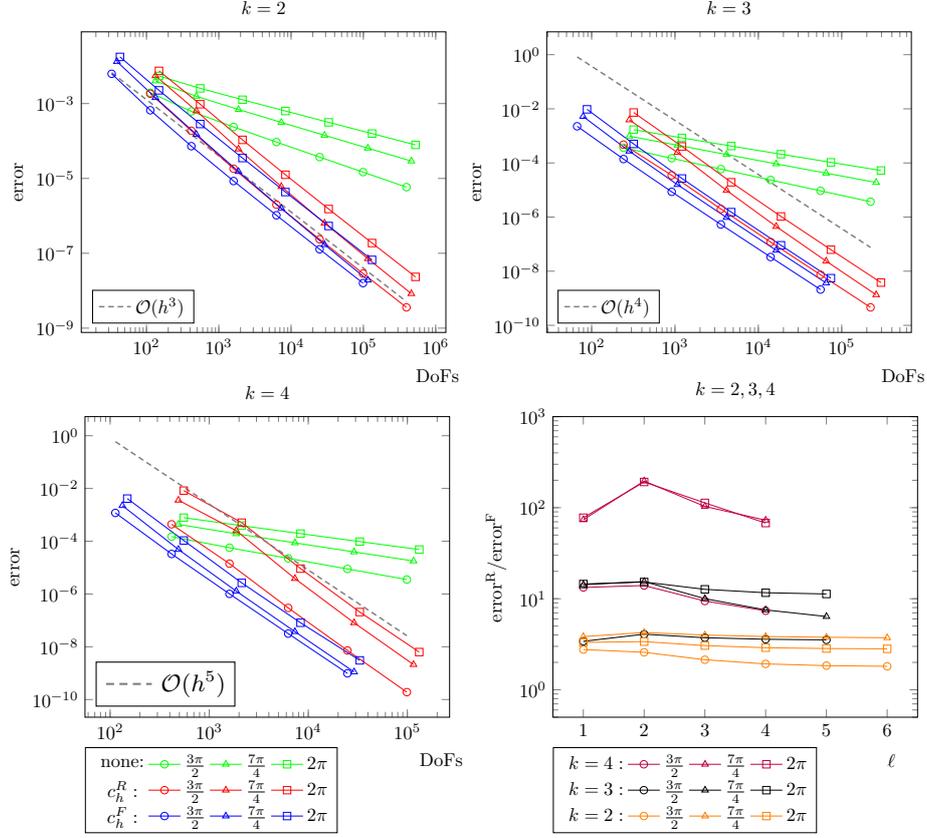

In \cref{fig:comp}, we graphically illustrate the comparison of the convergence rates of the standard scheme and the two corrections, for all studied scenarios, i.e., approximation orders $k = 2,3,4$ and $\omega = 3/2\pi,7/4\pi,2\pi$. The first three plots demonstrate the recovered optimal rates $k+1$ of the error $\|u-\uhm\|_{0,\alpha}$, $\alpha = k-\lambda_1$, where the $y$-label of the plots represents the error, i.e., $\|u-\uhm\|_{0,\alpha}$ and $\|u-u_h\|_{0,\alpha}$, respectively. As one can see, both corrections are qualitatively comparable, nevertheless, the function-correction $c_h^F(\cdot,\cdot)$ is quantitatively significantly better for higher $k$. This observation is depicted in the bottom right  plot of \cref{fig:comp}, showing the ratio between the errors caused by the approximations computed with $c_h^R(\cdot,\cdot)$ and $c_h^F(\cdot,\cdot)$. Note the clustering of the ratios by the number of the correction parameters $K$. 


\appendix
\section{Interpolation and approximation results}

At the end of this article we give some basic properties mandatory to classical approximation error analysis in the weighted norms. 
\begin{lemma}[Inverse inequality]\label{l:inv}
	For any $\vh\in\Vhk$ with $k\in\mathbb{N}$ and $\alpha>-1$ there holds
	\begin{align}\label{inv}
		\|v_h\|_{1,\alpha} \lesssim h^{-1} \|v_h\|_{0,\alpha}.
	\end{align}
\end{lemma}
\noindent The proof follows directly from a standard scaling argument.

\begin{lemma}[Higher-order interpolation in weighted-norms]\label{l:intest}
	Let $I_h^k : \mathcal{C}(\overline{\Omega}) \rightarrow \Vhk$ be the interpolation operator of order $k$. Then, for $v \in H^{k+1}_{\alpha}(\Omega)$ with $\alpha <  k$ it holds for $l\in\{0,1\}$:
	\begin{align}\label{interr}
		\|\g^l(v-\Ihk v)\|_{0,\beta} \lesssim h^{k+1-l+\beta-\alpha} \|v\|_{k+1,\alpha}, \qquad \alpha-(k+1)+l \leqslant \beta \leqslant \alpha.
	\end{align}
\end{lemma}

\begin{proof}
	The interpolation operator is for $v\in H^{k+1}_{\alpha}(\Omega)$ with $\alpha <k$ well-defined due to the continuous embedding \cref{embd2}. To prove \cref{interr}, we proceed in an element-wise fashion: For triangles $T\in\mathcal{T}_h$ not attached to the singular vertex we have on $T$ that $r \approx c_T$ for some constant $c_T>0$. Thus, for $l\in\{0,1\}$, by the definition of the weighted norm, standard interpolation error estimate and the fact that $c_T\gtrsim h$ we obtain for $\beta\leqslant \alpha$:
	\begin{align*}
		\| \nabla^l (v-\Ihk v)\|_{L^2_\beta(T)} &\lesssim c_T^{\beta} \|\nabla^l(v-\Ihk v)\|_{L^2(T)} \lesssim c_T^{\beta}\, h^{k+1-l}\|\nabla^{k+1} v\|_{L^2(T)}\\
		&\lesssim c_T^{\beta-\alpha}\, h^{k+1-l}\|\nabla^{k+1} v \|_{L^2_\alpha(T)}\lesssim h^{k+1-l+\beta-\alpha} \|v\|_{H^{k+1}_\alpha(T)}.
	\end{align*}
	If the singular point is a vertex of $T$, we derive the interpolation property on the reference element, i.e., by scaling arguments, Bramble--Hilbert lemma, embedding of $H^{k+1}_{\alpha}(\Omega)\embd H^{l}_{\beta}(\Omega)$ for $\beta\geqslant\alpha-(k+1)+l$ we can estimate:
	\begin{align*}
		\|\nabla^l (v-\Ihk v)\|_{L^2_\beta(T)}
		&\lesssim h^{\beta-l+1}\|\hat{r}^{\beta} \hat\nabla^l (\hat v-\hat I_h^k \hat v)\|_{L^2(\hat{T})}
		\lesssim h^{\beta-l+1} \|\hat{r}^{\alpha} \hat\nabla^{k+1} \hat v \|_{L^2(\hat{T})}\\
		&\lesssim h^{k+1-l+\beta-\alpha} \|r^{\alpha} \nabla^{k+1} v\|_{L^2(T)}
		\lesssim h^{k+1-l+\beta-\alpha} \|v\|_{H^{k+1}_\alpha(T)}.
	\end{align*}
	By summing over all elements we derive what has been claimed.
\end{proof}

\begin{lemma}[Interpolation of singular functions]\label{l:appsing}
	Let $\Ihk : \mathcal{C}(\overline{\Omega}) \rightarrow \Vhk$ be the interpolation operator of order $k$. Then, for singular functions $s_i$, $i\in\mathbb{N}$, it holds the interpolation error:
	\begin{align}\label{sinterr}
		\|\g^l(s_i-\Ihk s_i)\|_0 = \mathcal{O}(h^{1-l+\lambda_i}),
	\end{align}
	end the (standard) Galerkin approximation error:
	\begin{align}\label{saprerr}
		h^{(2-l)\lambda_i}\lesssim \|\nabla^l(s_i-s_{i,h})\|_0 \lesssim h^{(2-l)\lambda_i-\eps}.
	\end{align}
\end{lemma}
\begin{proof}
	The proof of \cref{saprerr} can be found in \cite{Blum1982}. Concerning \cref{sinterr}, the suboptimal estimate
	\begin{align*}
		h^{1-l+\lambda_i}\lesssim \|\nabla^l(s_i-\Ihk s_i)\|_0 \lesssim h^{1-l+\lambda_i-\eps}.
	\end{align*}
	follows directly from the regularity of the singular function $s_i$, respectively $s_i\not\in H^{k+1}_{k-\lambda_i}(\Omega)$, and \cref{l:intest}. The optimal estimate without $\eps-$dependence can be then derived by a direct computation of the interpolation error on the $h-$surrounding of the singular point and the complementary set to $\Omega$ where $s_i$ are already smooth.
\end{proof}

\begin{rem}\label{Iwhk}
Recalling that $(\Iwhk v - \Ihk v)\neq {0}$ on a $\mathcal{O}(h)$-surrounding of the singular point,
it is easy to see that the interpolation properties \cref{interr} also holds
for the restricted interpolation operator $\Iwhk$. Moreover, the interpolation error for the singular functions $s_i$ can be directly computated by the triangle inequality, \cref{interr} and a scaling argument
\begin{align}\label{estinttilde2}
	\|\nabla(s_i\!-\!\Iwhk s_i)\|_{0}\!\leqslant \!\|\nabla(s_i\!-\!\Ihk s_i)\|_{0}\! +\!\|\g(\Ihk s_i\!-\!\Iwhk s_i)\|_{0}
	\lesssim h^{\lambda_i}\! +\! \|s_i\|_{L^\infty(\Sh)}\!\lesssim\! h^{\lambda_i}.
\end{align}
\end{rem}

By standard estimates, we can easily obtain error bounds of the modified scheme, as stated in the next lemma. These, however, are for the pollution exhibiting $L^2-$norm in general sub-optimal, nevertheless they will be practical in the proof of \cref{mainthm}.
\begin{lemma}[A~priori estimates in standard norms]\label{l:standest}
	Let the solution $w$ of \cref{prob} be such that $w \in H^{k+1}_{\alpha}(\Omega) \cap H^1_0(\Omega)$ for some $0 \leqslant \alpha < k$, and $\whm\in\Vhk$ be its modified approximation. Then, for all $1-\lambda_1<\beta<1$ it is valid:
	\begin{align}\label{standest}
		\|\nabla(w-\whm)\|_0 \lesssim h^{k-\alpha} \|w\|_{k+1,\alpha} \quad \text{and} \quad \|w-\whm\|_0 \lesssim h^{k+1-\alpha-\beta}\|w\|_{k+1,\alpha}.
	\end{align}
\end{lemma}

\begin{proof}
	The modified bilinear form $a_h(\cdot,\cdot)$ is continuous and $H^1_0(\Omega)-$coercive by \cref{A}, thus, we employ the Strang lemma and modified Galerkin orthogonality \cref{GO}  and derive
	\begin{align*}
		\|w-\whm\|_1 &\lesssim \inf_{\vh\in\Vhk} \|w-\vh\|_1 + \sup_{0\neq\vh\in\Vhk} \frac{|a_h(w,\vh)-f(\vh)|}{\|\vh\|_1}\\
		&\lesssim\|w-\Ihk w\|_{1} + \|\nabla w\|_{{L^2(\Sh)}}.
	\end{align*}
	By embedding $H^{k+1}_{\alpha}(\Omega)\embd H^1_{\alpha-k}(\Omega)$ and the fact that $r\lesssim h$ on $\Sh$, we have for $\alpha<k$:
	\begin{align}\label{Shest}
		\|\nabla w\|_{L^2(\Sh)}\lesssim h^{k-\alpha} \|w\|_{H^1_{\alpha-k}(\Sh)} \lesssim h^{k-\alpha}\|w\|_{H^{k+1}_\alpha(\Omega)}.
	\end{align}
	This and the interpolation error \cref{interr} give the estimate for the gradient. For the $L^2-$estimate let us consider the dual problem
	\begin{align*}
		-\Delta v = (w-\whm) \quad \text{in }\Omega, \qquad \text{and} \qquad v=0 \quad \text{on }\partial \Omega,
	\end{align*}
	which is well-defined since $(w-\whm)\in L^{2}(\Omega)$, thus $v\in H^2_\beta(\Omega)$. Then, by the chain rule and the modified Galerkin orthogonality \cref{GO}, we obtain
	\begin{equation}\label{tmn}
		\begin{aligned}
			\|w-\whm\|_0^2 &=(w-\whm,w-\whm) = (w-\whm,-\Delta v) = a(w-\whm,v) \\
			&=a(w-\whm, v-I^1_h v) - c_h(\whm,I^1_h v)\\
			&\lesssim \|\nabla(w-\whm)\|_0 \|\nabla(v-I^1_h v)\|_{0} + \|\nabla \whm\|_{L^2(\Sh)} \|\nabla I^1_h v\|_{L^2(\Sh)}.
		\end{aligned}
	\end{equation}
	We consider both terms separately. By the interpolation results stated in \cref{interr}, the first part of \cref{standest}, and standard a~priori estimates for the Poisson equation $\|v\|_{2,\beta}\lesssim \|w-\whm\|_0$ we derive for $1-\lambda_1<\beta<1$:
	\begin{equation}\label{tm1}\begin{aligned}
		\|\nabla(w-\whm)\|_0 \|\nabla(v-I^1_h v)\|_0
		&\lesssim h^{k-\alpha} \|w\|_{k+1,\alpha}\, h^{1-\beta} \|v\|_{2,\beta}\\
		&\lesssim h^{k+1-\alpha-\beta} \|w\|_{k+1,\alpha} \|w-\whm\|_0.
	\end{aligned}\end{equation}
	For the second term of \cref{tmn}, we use similar scaling arguments as for \cref{Shest},  the triangle inequality and the above derived estimate for the gradients, thus
	\begin{align}\label{tm2}
		\|\nabla \whm\|_{L^2(\Sh)} 
		\lesssim h^{-\alpha}(\|\nabla(w-\whm)\|_{L^2_\alpha(\Sh)} + \|\nabla w\|_{L^2_\alpha(\Sh)})
		\lesssim h^{k-\alpha} \|w\|_{k+1,\alpha}.
	\end{align}
	By a~priori estimates of the Poisson equation with $v\in H^2_\beta(\Omega)$:
	\begin{equation}\label{tm3}\begin{aligned}
		\|\nabla I^1_h v\|_{L^2(\Sh)} &\lesssim \|\nabla (I^1_h v -v) \|_{L^2(\Sh)} + \|\nabla v \|_{L^2(\Sh)}\\
		&\lesssim h^{1-\beta} \|v\|_{2,\beta} \lesssim h^{1-\beta} \|w-\whm\|_0.
	\end{aligned}\end{equation}
	Combining the results of \cref{tm1}, \cref{tm2} and \cref{tm3} yields the $L^2-$estimate.
\end{proof}

If the solution $w$ has decreased regularity, let us say we know only that $w\in H^2_{\ta}(\Omega)$, $\ta>1-\lambda_1$, then we can show, independently of the approximation order $k$, by similar arguments:
\begin{align}\label{standest2}
	\|\nabla(w-\whm)\|_0 \lesssim h^{1-\ta} \|w\|_{2,\ta}.
\end{align}
This is a straightforward consequence of the modification of the scaling arguments in the interpolation error and \cref{Shest}.

\bibliographystyle{amsplain}
\bibliography{references}
\end{document}